\documentclass[a4paper]{amsart}

\usepackage[utf8]{inputenc}
\usepackage[english]{babel}

\usepackage[T1]{fontenc}
\usepackage{lmodern}  

\usepackage{my-math-symbol}
\usepackage{my-cleveref}
\usepackage{my-equation-numbering}
\usepackage{comment}
\usepackage{tikz}
\usetikzlibrary{cd}

\newcommand{\ovone}{\bar{1}}

\usepackage[non-sorted-cites,initials,alphabetic,nobysame]{amsrefs}

\AtBeginDocument{%
	\def\MR#1{}
}

\title{Non-embeddable torus and CR Paneitz operator}

\author{Pak Tung Ho}
\address{Department of Mathematics, Tamkang University Tamsui, New Taipei City 251301, Taiwan}
\email{paktungho@yahoo.com.hk}

\author{Yuya Takeuchi}
\address{Division of Mathematics \\ Institute of Pure and Applied Sciences \\ University of Tsukuba
	\\ 1- 1- 1 Tennodai, Tsukuba, Ibaraki 305-8571 Japan}
\email{ytakeuchi@math.tsukuba.ac.jp, yuya.takeuchi.math@gmail.com}



\thanks{The first author was supported  by the National Science and Technology Council (NSTC),
Taiwan, with grant Number: 114-2115-M-032-003-MY2.
The second author was supported by JSPS KAKENHI Grant Number JP21K13792 and JP25K17247.}

\begin{document}

\begin{abstract}
	The CR Paneitz operator is closely related to several important problems in CR geometry.
	In this paper,
	we study the CR Paneitz operator on non-embeddable three-dimensional tori.
	Under mild assumptions,
	we show that it possesses infinitely many negative eigenvalues.
	We also provide concrete examples satisfying the assumptions.
\end{abstract}

\maketitle

\section{Introduction}
\label{section:introduction}

The \emph{CR Paneitz operator} is a CR invariant linear differential operator
whose leading term is the square of the sub-Laplacian.
This plays a crucial role in three-dimensional CR geometry.
It is deeply connected with global embeddability~\cite{Chanillo-Chiu-Yang2012},
the CR positive mass theorem~\cite{Cheng-Malchiodi-Yang2017},
and the logarithmic singularity of the Szegő kernel~\cite{Hirachi1993}.

Let $(M, T^{1, 0} M, \theta)$ be a closed pseudo-Hermitian manifold of dimension three.
We consider the CR Paneitz operator $P$ on $M$
as an unbounded operator on $L^{2}(M)$ with domain $C^{\infty}(M)$,
which is closable.
In the embeddable case,
Hsiao~\cite{Hsiao2015} proved that $P$ is essentially self-adjoint and has closed range.
Moreover,
its spectrum $\Spec P$ is a discrete subset of $\bbR$ and consists only of eigenvalues.
Furthermore,
the eigenspace associated with each non-zero eigenvalue
is a finite-dimensional subspace of $C^{\infty}(M)$.
Later,
the second author~\cite{Takeuchi2020-Paneitz} showed that
the CR Paneitz operator $P$ on an embeddable CR manifold is non-negative.

Without assuming embeddability,
the second author~\cite{Takeuchi2024-preprint} proved that
$P$ is essentially self-adjoint,
$\Spec P \setminus \{0\}$ is a discrete subset of $\bbR \setminus \{0\}$ and consists only of eigenvalues,
and the eigenspace associated with each non-zero eigenvalue
is a finite-dimensional subspace of $C^{\infty}(M)$.
However,
much less is known than in the embeddable case.

In this paper,
we consider $\bbT^{2}$-invariant strictly pseudoconvex CR structures on the three-dimensional torus.
The embeddability of such CR manifolds was first considered by Barrett~\cite{Barrett1988}.
Later,
Dall'Ara and Son~\cite{DallAra-Son2023} introduced a canonical contact form in the embeddable case
and investigated an upper bound for the first positive eigenvalue of the Kohn Laplacian.
We extend their construction to the non-embeddable case
and study the spectrum of the CR Paneitz operator.

Let $L \in \bbR_{> 0}$
and let $(\alpha, \beta) \colon \bbR \to \bbC^{2}$ be an $L$-periodic smooth function
such that $(\Im \alpha)^{2} + (\Im \beta)^{2} = 1$.
Then there exists an $L$-periodic $\bbR$-valued function $\kappa$ such that
\begin{equation}
	(\Im \alpha', \Im \beta')
	= \kappa (\Im \beta, - \Im \alpha).
\end{equation}
Observe that $\kappa$ is precisely the curvature of the plane curve
\begin{equation}
	\gamma(s) = (\xi(s), \eta(s))
	\coloneqq \int_{0}^{s} (- \Im \alpha(t), \Im \beta(t)) \, d t.
\end{equation}
Assume that $\kappa$ is positive.
Then we can define a contact form $\theta$
on the three-dimensional torus $M \coloneqq (\bbR / L \bbZ) \times \bbT^{2}$ by
\begin{equation}
	\theta
	\coloneqq (\Im \alpha \Re \beta - \Re \alpha \Im \beta) \, d s
		+ \Im \beta \, d x - \Im \alpha \, d y.
\end{equation}
Let $Z_{1}$ be the complex vector field on $M$ defined by
\begin{equation}
	Z_{1}
	\coloneqq \frac{1}{\sqrt{2 \kappa}} \rbra*{\pdv{}{s} + \alpha \pdv{}{x} + \beta \pdv{}{y}}.
\end{equation}
Then $T^{1, 0} M \coloneqq \bbC Z_{1}$ is a $\bbT^{2}$-invariant strictly pseudoconvex CR structure on $M$.
Set
\begin{equation}
	A \coloneqq \exp \rbra*{- \sqrt{- 1} \int_{0}^{L} \ovxa(t) \, d t},
	\qquad
	B \coloneqq \exp \rbra*{- \sqrt{- 1} \int_{0}^{L} \ovxb(t) \, d t}.
\end{equation}
The CR manifold $(M, T^{1, 0} M)$ is embeddable
if and only if $A = B = 1$ and $\gamma$ is simple (\cref{prop:simple-is-equivalent-to-embeddable}).
If $(M, T^{1, 0} M)$ is non-embeddable and $\kappa$ satisfies an inequality,
then we can show that the CR Paneitz operator on $(M, T^{1, 0} M, \theta)$
has infinitely many negative eigenvalues.

\begin{theorem}
\label{thm:infinitely-many-negative-eigenvalues-for-non-simple-curve}
	Assume that $(A, B) \neq (1, 1)$ or $\gamma$ is not simple,
	and that $\kappa^{4} - \kappa \kappa'' + (\kappa')^{2} > 0$.
	Then the CR Paneitz operator $P$ on $(M, T^{1, 0} M, \theta)$ has infinitely many negative eigenvalues.
\end{theorem}

Note that the assumptions of this theorem are \emph{generic} (\cref{prop:perturbation-of-non-simple}).
We will also provide concrete examples satisfying the assumptions in this theorem,
including familiys of Lima\c{c}ons, hypotrochoids, and epitrochoids
(\cref{section:Limacon,section:Hypotrochoid,section:Epitrochoid}).

This paper is organized as follows.
In \cref{section:CR-manifolds},
we recall basic facts on CR manifolds and the definition of the CR Paneitz operator.
In \cref{section:CR-structures-on-torus},
we describe $\bbT^{2}$-invariant strictly pseudoconvex CR structures on the three-dimensional torus.
\cref{section:analysis-on-torus} is devoted to the proof of
\cref{thm:infinitely-many-negative-eigenvalues-for-non-simple-curve}.
In \cref{section:multiple-iterates},
we consider iterates of a smooth closed plane curve.
In \cref{section:Limacon,section:Hypotrochoid,section:Epitrochoid},
we show that Lima\c{c}on, hypotrochoid, and epitrochoid, respectively, satisfy
the assumptions of \cref{thm:infinitely-many-negative-eigenvalues-for-non-simple-curve}.
In \cref{section:concluding-remarks},
we propose some related problems and give some observations.

\section{CR manifolds}
\label{section:CR-manifolds}

Let $M$ be a smooth three-dimensional manifold without boundary.
A \emph{CR structure} is a complex line subbundle $T^{1, 0} M$
of the complexified tangent bundle $T M \otimes \bbC$ such that
\begin{equation}
	T^{1, 0} M \cap T^{0, 1} M = 0,
\end{equation}
where $T^{0, 1} M$ is the complex conjugate of $T^{1, 0} M$ in $T M \otimes \bbC$.
Introduce an operator $\delbb \colon C^{\infty}(M) \to \Gamma((T^{0, 1} M)^{\ast})$ by
\begin{equation}
	\delbb f = (d f)|_{T^{0, 1} M}.
\end{equation}
A smooth function $f$ is called a \emph{CR holomorphic function} if $\delbb f = 0$.
Denote by $\calO$ the space of CR holomorphic functions.
A \emph{CR pluriharmonic function} is a real-valued smooth function
that is locally the real part of a CR holomorphic function.
We denote by $\calP$ the space of CR pluriharmonic functions.
We introduce $\calP_{\bbC} \coloneqq \calP \otimes \bbC$ to be used later.
Note that $\calO, \overline{\calO} \subset \calP_{\bbC}$.
A CR manifold $(M, T^{1, 0} M)$ is said to be \emph{embeddable}
if there exists a smooth embedding $F = (f_{1}, \dots f_{N}) \colon M \to \bbC^{N}$
such that each $f_{i}$ is a CR holomorphic function.
Note that this condition is equivalent to $F_{\ast} (T^{1, 0} M) \subset T^{1, 0} \bbC^{N}$.

A CR structure $T^{1, 0} M$ is said to be \emph{strictly pseudoconvex}
if there exists a nowhere-vanishing real one-form $\theta$ on $M$
such that
$\theta$ annihilates $T^{1, 0} M$ and
\begin{equation}
	- \sqrt{- 1} d \theta (Z, \ovZ) > 0, \qquad
	0 \neq Z \in T^{1, 0} M;
\end{equation}
we call such a one-form a \emph{contact form}.
The triple $(M, T^{1, 0} M, \theta)$ is called a \emph{pseudo-Hermitian manifold}.
Denote by $T$ the \emph{Reeb vector field} with respect to $\theta$; 
that is, the unique vector field satisfying
\begin{equation}
	\theta(T) = 1, \qquad T \contr d \theta = 0.
\end{equation}
Let $Z_{1}$ be a local frame of $T^{1, 0} M$,
and set $Z_{\ovone} = \overline{Z_{1}}$.
Then
$(T, Z_{1}, Z_{\ovone})$ gives a local frame of $T M \otimes \bbC$,
called an \emph{admissible frame}.
Its dual frame $(\theta, \theta^{1}, \theta^{\ovone})$
is called an \emph{admissible coframe}.
The two-form $d \theta$ is written as
\begin{equation}
	d \theta = \sqrt{- 1} h_{1 \ovone} \theta^{1} \wedge \theta^{\ovone},
\end{equation}
where $h_{1 \ovone}$ is a positive function.
We use $h_{1 \ovone}$ and its multiplicative inverse $h^{1 \ovone}$
to raise and lower indices.

A contact form $\theta$ induces a canonical connection $\nabla$,
called the \emph{Tanaka-Webster connection} with respect to $\theta$.
It is defined by
\begin{equation}
	\nabla T
	= 0,
	\quad
	\nabla Z_{1}
	= \omega_{1}^{\ 1} Z_{1},
	\quad
	\nabla Z_{\ovone}
	= \omega_{\ovone}^{\ \ovone} Z_{\ovone},
\end{equation}
where $\omega_{\ovone}^{\ \ovone} = \overline{\omega_{1}^{\ 1}}$,
with the following structure equations:
\begin{equation}
	d \theta^{1}
	= \theta^{1} \wedge \omega_{1}^{\ 1}
	+ A^{1}_{\ \ovone} \theta \wedge \theta^{\ovone},
	\qquad
	d h_{1 \ovone}
	= \omega_{1}^{\ 1} h_{1 \ovone}
	+ h_{1 \ovone} \omega_{\ovone}^{\ \ovone}.
\end{equation}
The tensor $A_{1 1} = \overline{A_{\ovone \ovone}}$
is called the \emph{Tanaka-Webster torsion}.
The curvature form
$\Omega_{1}^{\ 1} = d \omega_{1}^{\ 1}$
of the Tanaka-Webster connection satisfies
\begin{equation}
\label{eq:curvature-form-of-TW-connection}
	\Omega_{1}^{\ 1}
	\equiv \Scal \cdot h_{1 \ovone} \theta^{1} \wedge \theta^{\ovone}
		\qquad \text{modulo } \theta,
\end{equation}
where $\Scal$ is the \emph{Tanaka-Webster scalar curvature}.
We denote the components of a successive covariant derivative of a tensor
by subscripts preceded by a comma,
for example, $K_{1 \ovone , 1}$;
we omit the comma if the derivatives are applied to a function.
We use the index $0$ for the component $T$ or $\theta$ in our index notation.
In this notation,
the operator $\delbb$ is given by
\begin{equation}
\label{eq:tensorial-rep-of-delb}
	\delbb f
	= f_{\ovone} \theta^{\ovone}.
\end{equation}
Define the \emph{Kohn Laplacian} $\Box_{b}$ and the \emph{sub-Laplacian} $\Delta_{b}$ by
\begin{equation}
\label{eq:Kohn-Laplacian}
	\Box_{b} f
	\coloneqq - f_{\ovone}^{\ \ovone},
	\qquad
	\Delta_{b} f
	\coloneqq (\Box_{b} + \overline{\Box}_{b}) f
\end{equation}
for $f \in C^{\infty}(M)$.
It is known that a closed strictly pseudoconvex CR manifold $(M, T^{1, 0} M)$ is embeddable
if and only if the Kohn Laplacian $\Box_{b}$ has closed range~\cites{Burns1979,Kohn1986}.

\begin{proposition}
\label{prop:criterion-of-embeddability}
	The point $0 \in \bbR$ is an accumulation point of $\Spec \Box_{b}$
	if there exists a positive real number $C$ such that
	$\Box_{b}$ has infinitely many positive eigenvalues less than or equal to $C$ with multiplicity.
\end{proposition}

\begin{proof}
	Suppose to the contrary that $0$ is isolated in $\Spec \Box_{b}$.
	Then there exists $\varepsilon > 0$ such that
	$\Spec \Box_{b} \setminus \{0\} \subset \clop{\varepsilon}{\infty}$.
	It follows from \cite{Burns-Epstein1990-Embed}*{Theorem 1.3} that
	\begin{equation}
		\Spec \Box_{b} \cap \opcl{0}{C}
		= \Spec \Box_{b} \cap \clcl{\varepsilon}{C}
	\end{equation}
	is a finite set and consists only of eigenvalues with finite multiplicity,
	which contradicts the assumption.
	This completes the proof.
\end{proof}

We next consider a CR analog $d^{c}_{\CR}$ of $d^{c}$,
which was introduced by \cite{Takeuchi2020-Paneitz}*{Section 3}.
The differential operator
\begin{equation}
\label{eq:definition-of-d^c_CR}
	d^{c}_{\CR} \colon C^{\infty}(M) \to \Omega^{1} (M) ; \quad
	u \mapsto \frac{\sqrt{- 1}}{2} \rbra*{u_{\ovone} \theta^{\ovone}
		- u_{1} \theta^{1}} + \frac{1}{2} (\Delta_{b} u) \theta
\end{equation}
is independent of the choice of $\theta$~\cite{Takeuchi2020-Paneitz}*{Lemma 3.1}.
Moreover,
\begin{equation}
\label{eq:formula-of-dd^c_CR}
	d d^{c}_{\CR} u
	= (P_{1} u) \theta \wedge \theta^{1}
		+ (P_{\ovone} u) \theta \wedge \theta^{\ovone},
\end{equation}
where
\begin{equation}
	P_{1} u
	= u_{\ovone \ 1}^{\ \ovone} + \sqrt{- 1} A_{1 1} u^{1},
	\qquad
	P_{\ovone} u
	= u_{1 \ \ovone}^{\ 1} - \sqrt{- 1} A_{\ovone \ovone} u^{\ovone}.
\end{equation}
In particular,
$u \in \calP_{\bbC}$ if and only if $d d^{c}_{\CR} u = 0$;
see \cite{Takeuchi2020-Paneitz}*{Lemma 3.2}.

\begin{lemma}
\label{lem:exact-sequence-for-CR-pluriharmonic}
	The operator $d^{c}_{\CR}$ induces the exact sequence
	\begin{equation}
		\calO \xrightarrow{\Re} \calP \xrightarrow{d^{c}_{\CR}} H^{1}(M; \bbR).
	\end{equation}
	In particular,
	one has $\dim_{\bbR} \calP / \Re \calO \leq \dim_{\bbR} H^{1}(M; \bbR)$.
\end{lemma}

\begin{proof}
	This result is a part of \cite{Case2025}*{Theorem 2.4.5},
	but we give an elementary proof for the reader's convenience.
	Assume that $u \in \calP$ satisfies $[d^{c}_{\CR} u] = 0$ in $H^{1}(M; \bbR)$.
	This means that there exists a real-valued $v \in C^{\infty}(M)$
	such that $2 d^{c}_{\CR} u = d v$.
	Comparing the $\theta^{\ovone}$ term implies that
	$\sqrt{- 1} u_{\ovone} = v_{\ovone}$.
	We obtain from this that $f = u + \sqrt{- 1} v$ is a CR holomorphic function
	satisfying $\Re f = u$.
\end{proof}

The \emph{CR Paneitz operator} $P$ is the fourth-order differential operator given by
\begin{equation}
	P
	\coloneqq \overline{\Box}_{b} \Box_{b} + \calQ,
\end{equation}
where
\begin{equation}
	\calQ u
	\coloneqq \sqrt{- 1} (A^{\ovone \ovone} u_{\ovone})_{, \ovone}.
\end{equation}
Note that $P u = (P_{1} u)_{,}^{\ 1} = (P_{\ovone} u)_{,}^{\ \ovone}$.
In particular,
$P$ annihilates $\calP_{\bbC}$.
This operator is real and formally self-adjoint;
see \cite{Gover-Graham2005}*{Proposition 5.1} for example.
Remark that our $P$ is just the operator $P_{0, 0}$ in this paper.
It is known that
the CR Paneitz operator is non-negative
if $(M, T^{1, 0} M)$ is embeddable~\cite{Takeuchi2020-Paneitz}*{Theorem 1.1}.
Conversely,
$(M, T^{1, 0} M)$ is embeddable
if the CR Paneitz operator is non-negative and the Tanaka-Webster scalar curvature is positive%
~\cite{Chanillo-Chiu-Yang2012}*{Theorem 1.4(a)}.

Analytic properties of the CR Paneitz operator have been studied by the second author.
It is essentially self-adjoint;
that is,
it has the unique self-adjoint extension~\cite{Takeuchi2024-preprint}*{Theorem 1.1}.
By abuse of notation,
we use the same symbol $P$ for this self-adjoint extension.
The spectrum $\Spec P$ is discrete except $0$
and consists only of eigenvalues.
Moreover,
the eigenspace of each non-zero eigenvalue is a finite dimensional subspace of $C^{\infty}(M)$%
~\cite{Takeuchi2024-preprint}*{Theorem 1.2}.

\section{$\bbT^{2}$-invariant CR structures on three-dimensional torus}
\label{section:CR-structures-on-torus}

Let $L \in \bbR_{> 0}$
and let $(\alpha, \beta) \colon \bbR \to \bbC^{2}$ be an $L$-periodic smooth function
such that $(\Im \alpha)^{2} + (\Im \beta)^{2} = 1$.
Then there exists an $L$-periodic $\bbR$-valued function $\kappa$ such that
\begin{equation}
	(\Im \alpha', \Im \beta')
	= \kappa (\Im \beta, - \Im \alpha).
\end{equation}
Observe that $\kappa$ is precisely the curvature of the plane curve
\begin{equation}
	\gamma(s) = (\xi(s), \eta(s))
	\coloneqq \int_{0}^{s} (- \Im \alpha(t), \Im \beta(t)) \, d t,
\end{equation}
which is parametrized by arc-length.
Assume that $\kappa$ is positive.
Then we can define a contact form $\theta$
on the three-dimensional manifold $M^{\prime} \coloneqq \bbR \times \bbT^{2}$ by
\begin{equation}
	\theta
	\coloneqq (\Im \alpha \Re \beta - \Re \alpha \Im \beta) \, d s
		+ \Im \beta \, d x - \Im \alpha \, d y.
\end{equation}
Note that
\begin{equation}
	d \theta
	= - \kappa (\Im \alpha \, d s \wedge d x + \Im \beta \, d s \wedge d y),
	\qquad
	\theta \wedge d \theta
	= \kappa \, d s \wedge d x \wedge d y
	> 0.
\end{equation}
The Reeb vector field $T$ with respect to $\theta$ is given by
\begin{equation}
	T
	= \Im \beta \pdv{}{x} - \Im \alpha \pdv{}{y}.
\end{equation}
Let $Z_{1}$ be the complex vector field on $M^{\prime}$ defined by
\begin{equation}
	Z_{1}
	\coloneqq \frac{1}{\sqrt{2 \kappa}} \rbra*{\pdv{}{s} + \alpha \pdv{}{x} + \beta \pdv{}{y}}.
\end{equation}
Then this satisfies
\begin{equation}
	\theta(Z_{1}) = 0,
	\qquad
	- \sqrt{- 1} d \theta(Z_{1}, Z_{\ovone})
	= 1 > 0.
\end{equation}
Hence $T^{1, 0} M^{\prime} \coloneqq \bbC Z_{1}$
is a $\bbT^{2}$-invariant strictly pseudoconvex CR structure on $M^{\prime}$.
The admissible coframe $(\theta, \theta^{1}, \theta^{\ovone})$ is given by
\begin{equation}
	\theta^{1}
	= \sqrt{\frac{\kappa}{2}} \rbra*{(1 + \sqrt{- 1} (\Re \alpha \Im \alpha + \Re \beta \Im \beta)) \, d s
		- \sqrt{- 1} \Im \alpha \, d x - \sqrt{- 1} \Im \beta \, d y}.
\end{equation}
These CR structure and the contact form on $M^{\prime}$ descends on
the three-dimensional torus $M \coloneqq (\bbR / L \bbZ) \times \bbT^{2}$.
In particular if $\gamma = (\xi, \eta)$ is a smooth closed plane curve parametrized by arc-length
with length $L$ and with positive curvature $\kappa$,
then $(\alpha, \beta) = (- \sqrt{- 1} \xi', \sqrt{- 1} \eta')$
defines the strictly pseudoconvex CR structure and contact form on $M$.
Note that our construction is inspired by a paper of Dall'Ara and Son~\cite{DallAra-Son2023}.

We next compute the Tanaka-Webster connection.
\begin{align}
	d \theta^{1}
	&= \frac{\kappa'}{\sqrt{8 \kappa}}
		\, d s \wedge \rbra*{- \sqrt{- 1} \Im \alpha \, d x - \sqrt{- 1} \Im \beta \, d y} \\
	&\quad + \sqrt{\frac{\kappa}{2}} \rbra*{- \sqrt{- 1} \kappa \Im \beta \, d s \wedge d x
		+ \sqrt{- 1} \kappa \Im \alpha \, d s \wedge d y} \\
	&= - \frac{\kappa'}{\sqrt{8 \kappa^{3}}} \theta^{1} \wedge \theta^{\ovone}
		+ \theta \wedge \rbra*{\frac{\sqrt{- 1}}{2} \kappa (\theta^{1} + \theta^{\ovone})} \\
	&= \theta^{1} \wedge \rbra*{\frac{\kappa'}{\sqrt{8 \kappa^{3}}} (\theta^{1} - \theta^{\ovone})
		- \frac{\sqrt{- 1}}{2} \kappa \theta}
		+ \theta \wedge \rbra*{\frac{\sqrt{- 1}}{2} \kappa \theta^{\ovone}}.
\end{align}
This yields that $A^{1}_{\ \ovone} = (\sqrt{- 1} / 2) \kappa$ and
\begin{align}
	\omega_{1}^{\ 1}
	&= \frac{\kappa'}{\sqrt{8 \kappa^{3}}} (\theta^{1} - \theta^{\ovone})
		- \frac{\sqrt{- 1}}{2} \kappa \theta \\
	&= \frac{\sqrt{- 1} \kappa'}{2 \kappa}
		\rbra*{(\Re \alpha \Im \alpha + \Re \beta \Im \beta) \, d s
		- \Im \alpha \, d x - \Im \beta \, d y} - \frac{\sqrt{- 1}}{2} \kappa \theta.
\end{align}
Then
\begin{align}
	d \omega_{1}^{\ 1}
	&= \frac{\sqrt{- 1} (\log \kappa)''}{2}
		\rbra*{- \Im \alpha \, d s \wedge d x - \Im \beta \, d s \wedge d y} \\
	&\quad + \frac{\sqrt{- 1} \kappa'}{2 \kappa}
		\rbra*{- \kappa \Im \beta \, d s \wedge d x + \kappa \Im \alpha \, d s \wedge d y} \\
	&\quad - \frac{\sqrt{- 1}}{2} \kappa' \, d s \wedge \theta
		- \frac{\sqrt{- 1}}{2} \kappa \, d \theta \\
	&= \rbra*{\frac{\kappa}{2} - \frac{(\log \kappa)''}{2 \kappa}} \theta^{1} \wedge \theta^{\ovone}
		- \frac{\sqrt{- 1} \kappa'}{\sqrt{2 \kappa}} (\theta^{1} + \theta^{\ovone}) \wedge \theta.
\end{align}
This implies
\begin{equation}
\label{eq:TW-scalar-curvature-on-torus}
	\Scal
	= \frac{\kappa}{2} - \frac{(\log \kappa)''}{2 \kappa}
	= \frac{\kappa^{4} - \kappa \kappa'' + (\kappa')^{2}}{2 \kappa^{3}}.
\end{equation}

\section{Analysis on three-dimensional torus}
\label{section:analysis-on-torus}

Let $L \in \bbR_{> 0}$
and let $(\alpha, \beta) \colon \bbR \to \bbC$ be an $L$-periodic smooth function
such that $(\Im \alpha)^{2} + (\Im \beta)^{2} = 1$ and $\kappa$ is positive.
Consider the pseudo-Hermitian manifold $(M^{\prime}, T^{1, 0} M^{\prime}, \theta)$ as in the previous section.
Set
\begin{equation}
	H_{m, n}^{\prime}
	\coloneqq \Set{f \in C^{\infty}(M^{\prime}) | f(s, x, y) = g(s) e^{\sqrt{- 1} (m x + n y)}}
\end{equation}
for $(m, n) \in \bbZ^{2}$.
For any $f(s, x, y) = g(s) e^{\sqrt{- 1} (m x + n y)} \in H_{m, n}^{\prime}$,
we have
\begin{gather}
	Z_{1} f
	= \frac{1}{\sqrt{2 \kappa}} (g' + \sqrt{- 1} (m \alpha + n \beta) g) e^{\sqrt{- 1} (m x + n y)}, \\
	Z_{\ovone} f
	= \frac{1}{\sqrt{2 \kappa}} (g' + \sqrt{- 1} (m \ovxa + n \ovxb) g) e^{\sqrt{- 1} (m x + n y)}.
\end{gather}
This implies that $f$ is CR holomorphic if and only if
\begin{equation}
	g(s)
	= C \exp \rbra*{- \sqrt{- 1} \int_{0}^{s} (m \ovxa(t) + n \ovxb(t)) \, d t}
\end{equation}
for some constant $C \in \bbC$.
To simplify notation,
we set
\begin{equation}
\label{eq:CR-hol-in-Hmn}
	f_{m, n}(s, x, y)
	\coloneqq \exp \rbra*{- \sqrt{- 1} \int_{0}^{s} (m \ovxa(t) + n \ovxb(t)) \, d t}
		e^{\sqrt{- 1} (m x + n y)} \in \calO \cap H_{m, n}^{\prime}.
\end{equation}
Note that $\abs{f_{m, n}} = e^{m \xi - n \eta}$.

\begin{proposition}
	One has
	\begin{equation}
		\calP_{\bbC} \cap H_{m, n}^{\prime}
		= \begin{cases}
			\Span_{\bbC} \{ 1, \xi, \eta \}, & (m, n) = (0, 0), \\
			\Span_{\bbC} \{ f_{m, n}, \ovf_{- m, - n} \}, & (m, n) \neq (0, 0).
		\end{cases}
	\end{equation}
\end{proposition}

\begin{proof}
	We first consider the case of $(m, n) = (0, 0)$.
	If $g(s) \in H_{0, 0}^{\prime}$,
	then
	\begin{align}
		P_{1} g
		&= Z_{1} \rbra*{Z_{1} Z_{\ovone} g - \omega_{\ovone}^{\ \ovone}(Z_{1}) Z_{\ovone} g}
			+ \sqrt{- 1} A^{\ovone \ovone} Z_{\ovone} g \\
		&= \frac{1}{\sqrt{2 \kappa}} \dv{}{s} \rbra*{\frac{1}{2 \kappa} g''}
			+ \sqrt{- 1} \rbra*{- \frac{\sqrt{- 1}}{2} \kappa} \frac{1}{\sqrt{2 \kappa}} g' \\
		&= \frac{1}{\sqrt{8 \kappa^{3}}} (g''' - \kappa' \kappa^{- 1} g'' + \kappa^{2} g').
	\end{align}
	This means that $g \in \calP_{\bbC} \cap H_{0, 0}^{\prime}$ if and only if
	\begin{equation}
	\label{eq:CR-ph-and-ODE}
		g''' - \kappa' \kappa^{- 1} g'' + \kappa^{2} g' = 0,
	\end{equation}
	which is a third-order explicit linear ordinary differential equation.
	Hence the dimension of $\calP_{\bbC} \cap H_{0, 0}^{\prime}$ is at most $3$.
	It is easy to see that a constant function is a solution of \cref{eq:CR-ph-and-ODE}.
	We obtain from $\xi'' = - \kappa \eta'$ and $\eta'' = \kappa \xi'$ that
	$\xi$ and $\eta$ are also solutions of \cref{eq:CR-ph-and-ODE}.
	Since $1, \xi, \eta$ are linearly independent,
	we have $\calP_{\bbC} \cap H_{0, 0}^{\prime} = \Span_{\bbC} \{ 1, \xi, \eta \}$.
	
	Now consider the case of $(m, n) \neq (0, 0)$.
	It follows from $f_{m, n}, f_{- m, - n} \in \calO$
	that $f_{m, n}, \ovf_{- m, - n} \in \calP_{\bbC} \cap H_{m, n}^{\prime}$.
	On the other hand,
	\cref{lem:exact-sequence-for-CR-pluriharmonic} implies
	$\dim_{\bbR} \calP / \Re \calO \leq \dim_{\bbR} H^{1}(M^{\prime}; \bbR) = 2$,
	and the above argument yields that $[d^{c} \xi]$ and $[d^{c} \eta]$ are linearly independent
	in $H^{1}(M^{\prime}; \bbR)$.
	Hence we have
	\begin{equation}
		\calP_{\bbC} \cap H_{m, n}^{\prime}
		= (\calO + \overline{\calO}) \cap H_{m, n}^{\prime}
		= \Span_{\bbC} \{ f_{m, n}, \ovf_{- m, - n} \},
	\end{equation}
	which completes the proof.
\end{proof}

We next consider the three-dimensional torus $(M, T^{1, 0} M, \theta)$.
Set
\begin{equation}
	H_{m, n}
	\coloneqq L^{2}(\bbR / L \bbZ, 4 \pi^{2} \kappa \, d s) e^{\sqrt{- 1} (m x + n y)}
	\subset L^{2}(M, \theta \wedge d \theta)
\end{equation}
for $(m, n) \in \bbZ^{2}$.
Then we have the following orthogonal decomposition:
\begin{equation}
	L^{2}(M, \theta \wedge d \theta)
	= \bigoplus_{m, n} H_{m, n}.
\end{equation}
Note that $Z_{1}$, $Z_{\ovone}$, $\Box_{b}$, and $P$ map $H_{m, n}$ to itself.
The space of CR holomorphic (resp.\ CR pluriharmonic) functions on $M$ can be identified with
the space of $L$-periodic CR holomorphic (resp.\ CR pluriharmonic) functions on $M^{\prime}$.
Set
\begin{equation}
	A \coloneqq \exp \rbra*{- \sqrt{- 1} \int_{0}^{L} \ovxa(t) \, d t},
	\qquad
	B \coloneqq \exp \rbra*{- \sqrt{- 1} \int_{0}^{L} \ovxb(t) \, d t}.
\end{equation}
Note that $\xi$ (resp.\ $\eta$) is $L$-periodic if and only if $\abs{A} = 1$ (resp.\ $\abs{B} = 1$).
The CR holomorphic function $f_{m, n} \in \calO \cap H_{m, n}^{\prime}$ is $L$-periodic
if and only if $A^{m} B^{n} = 1$.
Set
\begin{equation}
	\Lambda \coloneqq \Set{(m, n) \in \bbZ^{2} | A^{m} B^{n} = 1},
\end{equation}
which is a subgroup of $\bbZ^{2}$.
The above observation implies
\begin{proposition}
	One has
	\begin{equation}
		\calO \cap H_{m, n}
		= \begin{cases}
			\Span_{\bbC} \{ f_{m, n} \}, & (m, n) \in \Lambda, \\
			0, & (m, n) \notin \Lambda,
		\end{cases}
	\end{equation}
	for any $(m, n) \in \bbZ^{2}$.
	Moreover, one has
	\begin{equation}
		\calP_{\bbC} \cap H_{0, 0}
		= \begin{cases}
			\Span_{\bbC} \{ 1, \xi, \eta \}, & \abs{A} = 1 = \abs{B}, \\
			\Span_{\bbC} \{ 1, \xi \}, & \abs{A} = 1 \neq \abs{B}, \\
			\Span_{\bbC} \{ 1, \eta \}, & \abs{A} \neq 1 = \abs{B}, \\
			\bbC, & \abs{A} \neq 1 \neq \abs{B}.
		\end{cases}
	\end{equation}
	If $(m, n) \neq (0, 0)$,
	then
	\begin{equation}
		\calP_{\bbC} \cap H_{m, n}
		= \begin{cases}
			\Span_{\bbC} \{ f_{m, n}, \ovf_{- m, - n} \}, & (m, n) \in \Lambda, \\
			0, & (m, n) \notin \Lambda.
		\end{cases}
	\end{equation}
\end{proposition}

Define an additive subgroup $\Lambda^{\prime}$ of $\bbR^{2}$ by
\begin{equation}
	\Lambda^{\prime}
	= \Set{(\mu, \nu) \in \bbR^{2} | m \mu + n \nu \in \bbZ \text{ for any $(m, n) \in \Lambda$}},
\end{equation}
which contains $\bbZ^{2}$.
Note that $\Lambda^{\prime} = \bbZ^{2}$ if and only if $A = B = 1$.

\begin{proposition}[\cite{Barrett1988}]
\label{prop:simple-is-equivalent-to-embeddable}
	The CR manifold $(M, T^{1, 0} M)$ is embeddable
	if and only if $A = B = 1$ and $\gamma$ is simple.
\end{proposition}

\begin{proof}
	If $A = B = 1$,
	then the plane curve $\gamma$ is a smooth closed curve.
	In addition if $\gamma$ is simple,
	the map
	\begin{equation}
		F \colon M \to \bbC^{2}; \qquad (s, x, y) \mapsto (f_{1, 0}(s, x, y), f_{0, 1}(s, x, y))
	\end{equation}
	gives an embedding of $(M, T^{1, 0} M)$ to $\bbC^{2}$.
	If there exists $(\mu, \nu) \in \Lambda^{\prime} \setminus \bbZ^{2}$,
	then we cannot separate the points $(s, x, y)$ and $(s, x + 2 \pi \mu, y + 2 \pi \nu)$
	by CR holomorphic functions.
	If $A = B = 1$ but $\gamma$ is not simple,
	then $\gamma(s_{1}) = \gamma(s_{2})$ for some $0 \leq s_{1} < s_{2} < L$.
	Hence we cannot separate the points $(s_{1}, x, y)$ and $(s_{2}, x, y)$ by CR holomorphic functions.
\end{proof}

Consider $\Box_{b}$ on $H_{m, n}$.
If $u(s, x, y) = g(s) e^{\sqrt{- 1} (m x + n y)} \in H_{m, n}$ is smooth,
then
\begin{align}
	\Box_{b} u
	&= - Z_{1} Z_{\ovone} u + \omega_{\ovone}^{\ \ovone}(Z_{1}) Z_{\ovone} u \\
	&= - Z_{1} \rbra*{\frac{1}{\sqrt{2 \kappa}}
		(g' + \sqrt{- 1} (m \ovxa + n \ovxb) g) e^{\sqrt{- 1} (m x + n y)}} \\
	&\quad - \frac{\kappa'}{4 \kappa^{2}} (g' + \sqrt{- 1} (m \ovxa + n \ovxb) g) e^{\sqrt{- 1} (m x + n y)} \\
	&= - \frac{1}{2 \kappa} (g'' + 2 \sqrt{- 1} (m \Re \alpha + n \Re \beta) g' \\
	&\qquad \qquad + (\sqrt{- 1} (m \ovxa' + n \ovxb') - \abs{m \alpha + n \beta}^{2}) g)e^{\sqrt{- 1} (m x + n y)}.
\end{align}
Define a differential operator $A_{m, n}$ on $\bbR / L \bbZ$ by
\begin{equation}
	A_{m, n}
	\coloneqq - \frac{1}{2 \kappa} \rbra*{\frac{d^{2}}{d s^{2}}
		+ 2 \sqrt{- 1} (m \Re \alpha + n \Re \beta) \frac{d}{d s}
		+ \sqrt{- 1} (m \ovxa' + n \ovxb') - \abs{m \alpha + n \beta}^{2}}.
\end{equation}
Note that
\begin{equation}
	\int_{0}^{L} (A_{m, n} g) \ovg \, 4 \pi^{2} \kappa \, d s
	= \int_{0}^{L} \abs*{\frac{1}{\sqrt{2 \kappa}} (g'
		+ \sqrt{- 1} (m \ovxa + n \ovxb) g)}^{2} 4 \pi^{2} \kappa \, d s.
\end{equation}
In particular,
$A_{m, n}$ is a non-negative self-adjoint elliptic operator on $L^{2}(\bbR / L \bbZ, 4 \pi^{2} \kappa \, d s)$,
and it has the first positive eigenvalue $\lambda_{m, n}$.
Note that the variational characterization of the first positive eigenvalue implies that
\begin{equation}
\label{eq:variational-characterization}
	\lambda_{m, n}
	= \inf \Set{\frac{\norm{Z_{\ovone} u}_{L^{2}}^{2}}{\norm{u}_{L^{2}}^{2}} |
		u \in C^{\infty}(M) \cap H_{m, n}, u \perp \calO}.
\end{equation}

\begin{proof}[Proof of \cref{thm:infinitely-many-negative-eigenvalues-for-non-simple-curve}]
	It follows from \cref{prop:simple-is-equivalent-to-embeddable}
	that $(M, T^{1, 0} M)$ is non-embeddable.
	The assumption $\kappa^{4} - \kappa \kappa'' + (\kappa')^{2} > 0$
	and \cref{eq:TW-scalar-curvature-on-torus} imply that
	the Tanaka-Webster scalar curvature $\Scal$ of $(M, T^{1, 0} M, \theta)$ is positive.
	Set
	\begin{equation}
		\Sigma
		\coloneqq \Set{(m, n) \in \bbZ^{2} | \lambda_{m, n} < \frac{1}{2} \min \Scal}.
	\end{equation}
	Suppose to the contrary that $\# \Sigma < \infty$.
	Set
	\begin{equation}
		\mu
		\coloneqq \min_{(m, n) \in \Sigma} \lambda_{m, n}
		> 0.
	\end{equation}
	Since $\Spec \Box_{b} = \bigcup_{m, n} \Spec A_{m, n}$,
	we have $\Spec \Box_{b} \setminus \{0\} \subset \clop{\mu}{\infty}$,
	which means that $\Box_{b}$ has closed range.
	This contradicts the fact that $(M, T^{1, 0} M)$ is non-embeddable.
	Thus we have $\# \Sigma = \infty$.
	
	Take an eigenfunction $u_{m, n} \in H_{m, n}$ of $\Box_{b}$ associated with $\lambda_{m, n}$
	for each $(m, n) \in \Sigma$.
	The proof of \cite{Chanillo-Chiu-Yang2012}*{Theorem 1.4} implies that
	\begin{equation}
		\int_{M} (P u_{m, n}) \ovu_{m, n} \, \theta \wedge d \theta
		\leq - \lambda_{m, n} (\min \Scal - 2 \lambda_{m, n}) \norm{u_{m, n}}_{L^{2}}^{2}
		< 0.
	\end{equation}
	This yields that $P|_{H_{m, n}}$ has at least one negative eigenvalue.
	It follows from $\# \Sigma = \infty$ that
	$P$ has infinitely many negative eigenvalues.
\end{proof}

Before the end of this section,
we note that the assumptions in \cref{thm:infinitely-many-negative-eigenvalues-for-non-simple-curve} is \emph{generic}.
To this end,
we use the following simple but useful lemma.

\begin{lemma}
\label{lem:criterion-for-simple-curve}
	A smooth closed plane curve $\gamma$ with $\kappa > 0$ is simple
	if and only if $\int_{0}^{L} \kappa(s) \, d s = 2 \pi$.
	In particular,
	any $C^{2}$-small deformation of a non-simple closed curve with positive curvature
	is also not simple.
\end{lemma}

\begin{proof}
	It is known that the integral $\int_{0}^{L} \kappa(s) \, d s$ of the curvature $\kappa(s)$
	is equal to $2 \pi$ times the \emph{rotation index} of $\gamma$,
	which is an integer.
	The Hopf Umlaufsatz implies that
	if $\gamma$ is simple and $\kappa > 0$,
	then the rotation index of $\gamma$ must be $1$,
	and so $\int_{0}^{L} \kappa(s) \, d s = 2 \pi$;
	see \cite{Montiel-Ros2009}*{Theorem 9.50} for example.
	Conversely if
	\begin{equation}
		\int_{0}^{L} \abs{\kappa(s)} \, d s
		= \int_{0}^{L} \kappa(s) \, d s
		= 2 \pi,
	\end{equation}
	then $\gamma$ is simple,
	which is a special case of Fenchel's theorem;
	see \cite{Kobayashi2021}*{Theorem 1.3.2} for example.
	
	For the latter statement,
	let $\whxg$ be a closed curve sufficiently close to non-simple $\gamma$ in the $C^{2}$-topology.
	The curvature $\whxk$ of $\whxg$ is also positive
	and its integral is close to $\int_{0}^{L} \kappa(s) \, d s$,
	which is greater than or equal to $4 \pi$.
	As a result,
	$\whxg$ is also not simple.
\end{proof}

\begin{proposition}
\label{prop:perturbation-of-non-simple}
	Assume that $(\alpha, \beta)$ satisfies the assumptions
	in \cref{thm:infinitely-many-negative-eigenvalues-for-non-simple-curve}.
	Then any $C^{3}$-small perturbation of $(\alpha, \beta)$ also satisfies the assumptions.
\end{proposition}

\begin{proof}
	Let $(\whxa, \whxb)$ be an $L$-periodic smooth function
	sufficiently close to $(\alpha, \beta)$ in the $C^{3}$-topology.
	If $(A, B) \neq (1, 1)$,
	then $(\whA, \whB) \neq (1, 1)$.
	If $(A, B) = (1, 1)$ but $\gamma$ is not simple,
	then it follows from \cref{lem:criterion-for-simple-curve} that
	$\whxg$ is also not simple.
	Moreover,
	the curvature $\whxk$ of $\whxg$ is close to $\kappa$ in the $C^{2}$-topology;
	in particular,
	$\whxk$ is positive.
	Furthermore,
	$\whxk^{4} - \whxk \whxk'' + (\whxk')^{2}$ is close to $\kappa^{4} - \kappa \kappa'' + (\kappa')^{2}$
	in the $C^{0}$-topology,
	which implies $\whxk^{4} - \whxk \whxk'' + (\whxk')^{2} > 0$.
\end{proof}

\section{Example 1: Iterates of closed curves}
\label{section:multiple-iterates}

Fix a smooth closed curve $\gamma = (\xi, \eta) \colon \clcl{0}{L} \to \bbR^{2}$
parametrized by arc-length $s$
such that its curvature $\kappa(s)$ is positive.
We extend $\xi$, $\eta$, and $\kappa$ as $L$-periodic functions on $\bbR$.
For a positive integer $k$,
consider the smooth closed curve
\begin{equation}
	\gamma_{k} = (\xi, \eta) \colon \clcl{0}{k L} \to \bbR^{2},
\end{equation}
which is parametrized by arc-length.
Note that the curvature of $\gamma_{k}$ is equal to $\kappa$.
Consider the pseudo-Hermitian manifold $(M_{k}, T^{1, 0} M_{k}, \theta_{k})$ with respect to $\gamma_{k}$.
If $k \geq 2$,
then $(M_{k}, T^{1, 0} M_{k})$ is non-embeddable
since $\gamma_{k}$ is not simple.

\begin{theorem}
\label{thm:infinitely-many-negative-eigenvalues-for-k-fold-cover}
	Assume that $k \geq 2$
	and the curvature $\kappa$ of $\gamma$ satisfies $\kappa^{4} - \kappa \kappa'' + (\kappa')^{2} > 0$.
	Let $\whxg$ be a sufficiently $C^{4}$-small perturbation of $\gamma_{k}$
	and denote by $(\whM, T^{1, 0} \whM, \whxth)$ the pseudo-Hermitian manifold
	with respect to the curve $\whxg$.
	Then the CR Paneitz operator on $(\whM, T^{1, 0} \whM, \whxth)$
	has infinitely many negative eigenvalues.
	In particular,
	so does that on $(M_{k}, T^{1, 0} M_{k}, \theta_{k})$.
\end{theorem}

\begin{proof}
	It follows from \cref{prop:perturbation-of-non-simple} that
	$\whxg$ satisfies the assumptions in \cref{thm:infinitely-many-negative-eigenvalues-for-non-simple-curve}.
	Hence the CR Paneitz operator on $(\whM, T^{1, 0} \whM, \whxth)$
	has infinitely many negative eigenvalues.
\end{proof}

Moreover,
we can identify the kernel of the CR Paneitz operator on $(M_{k}, T^{1, 0} M_{k}, \theta_{k})$
when $\gamma$ is simple.

\begin{theorem}
\label{thm:kernel-of-CR-Paneitz}
	Assume that $\gamma$ is simple.
	The CR Paneitz operator $P$ on $(M_{k}, T^{1, 0} M_{k}, \theta_{k})$ satisfies
	$\Ker P \cap H_{m, n} = \calP_{\bbC} \cap H_{m, n}$ for any $(m, n) \in \bbZ^{2}$.
\end{theorem}

\begin{proof}
	We first note that $\Ker P \cap H_{m, n} \subset C^{\infty}(M_{k})$
	since the equation $P u = 0$ is reduced to a fourth-order regular linear ordinary differential equation
	with smooth coefficients.
	Take any $u \in \Ker P \cap H_{m, n}$.
	There exists $g_{1}(s) \in C^{\infty}(\bbR / k L \bbZ)$ such that $P_{1} u = g_{1} f_{m, n}$.
	Then we have
	\begin{align}
		0
		= P u
		= (P_{1} u)_{,}^{\ 1}
		&= Z_{\ovone} (g_{1} f_{m, n}) - \omega_{1}^{\ 1}(Z_{\ovone}) g_{1} f_{m, n} \\
		&= \frac{1}{\sqrt{2}} \rbra*{\kappa^{- 1 / 2} g_{1}'
			+ \frac{1}{2} \kappa' \kappa^{- 3 / 2} g_{1}} f_{m, n} \\
		&= \frac{1}{\sqrt{2} \kappa} (\kappa^{1 / 2} g_{1})' f_{m, n}.
	\end{align}
	This implies that $P_{1} u = C_{1} \kappa^{- 1 / 2} f_{m, n}$ for some constant $C_{1} \in \bbC$.
	A similar argument shows that
	$P_{\ovone} u = C_{\ovone} \kappa^{- 1 / 2} \ovf_{- m, - n}$ for some constant $C_{\ovone} \in \bbC$.
	
	Since $\gamma$ is a simple closed curve,
	$(M_{1}, T^{1, 0} M_{1})$ is embeddable.
	The function
	\begin{equation}
	\label{eq:sum-of-CR-pluriharmonic}
		u_{0}(s, x, y)
		\coloneqq \sum_{l = 0}^{k - 1} u(s + l L, x, y)
	\end{equation}
	descends to a function on $M_{1}$
	and is annihilated by the CR Paneitz operator on $M_{1}$,
	which implies $u_{0} \in \calP_{\bbC}$.
	Applying $P_{1}$ to \cref{eq:sum-of-CR-pluriharmonic},
	we obtain $k C_{1} \kappa^{- 1 / 2} f_{m, n} = 0$,
	and so $P_{1} u = 0$.
	Similarly,
	we have $P_{\ovone} u = 0$.
	Therefore $u \in \calP_{\bbC}$.
\end{proof}

\begin{example}
	Consider the unit circle $\gamma(s) = (\cos s, \sin s)\ (s \in \clcl{0}{2 \pi})$ in $\bbR^{2}$.
	Its curvature $\kappa$ is equal to $1$,
	and so
	\begin{equation}
		\kappa^{4} - \kappa \kappa'' + (\kappa')^{2}
		= 1 > 0.
	\end{equation}
	It follows from \cref{thm:infinitely-many-negative-eigenvalues-for-k-fold-cover}
	that for any $k \geq 2$,
	the CR Paneitz operator $P$ on $(M_{k}, T^{1, 0} M_{k}, \theta_{k})$ has infinitely many negative eigenvalues.
	Moreover,
	\cref{thm:kernel-of-CR-Paneitz} implies that $\Ker P \cap C^{\infty}(M)$ is dense in $\Ker P$.
\end{example}

As we noted,
$(M_{k}, T^{1, 0} M_{k})$ is non-embeddable for any $k \geq 2$
since $\gamma_{k}$ is not simple.
We can show that $0$ is an accumulation point of the spectrum of the Kohn Laplacian,
which yields another proof of the non-embeddability.

\begin{theorem}
	Assume that $k \geq 2$.
	Then the smallest positive eigenvalue $\lambda_{1}(\Box_{b}|_{H_{m, n}})$
	of $\Box_{b}|_{H_{m, n}}$ satisfies
	\begin{equation}
	\label{eq:upper-bound-of-Kohn-Laplacian-on-k-fold-cover}
		\lambda_{1}(\Box_{b}|_{H_{m, n}})
		\leq \frac{2 \pi^{2}}{k^{2} L^{2} \min \kappa}.
	\end{equation}
	In particular,
	$0$ is an accumulation point of $\Spec \Box_{b}$.
\end{theorem}

\begin{proof}
	Let $f_{m, n} \in \calO \cap H_{m, n}$ be as in \cref{eq:CR-hol-in-Hmn}
	and define
	\begin{equation}
		u_{m, n}(s, x, y)
		= e^{2 \pi \sqrt{- 1} s / k L} f_{m, n}(s, x, y) \in H_{m, n}.
	\end{equation}
	Then
	\begin{equation}
		\iproduct{f_{m, n}}{u_{m, n}}_{L^{2}}
		= 4 \pi^{2} \int_{0}^{k L} e^{- 2 \pi \sqrt{- 1} s / k L} \kappa e^{2 (m \xi - n \eta)} \, d s.
	\end{equation}
	Since $e^{- 2 \pi \sqrt{- 1} s / k L}$ is periodic with minimum period $k L$
	and $\kappa e^{2 (m \xi - n \eta)}$ is periodic with period $L$,
	this integral is equal to zero;
	here we use the assumption $k \geq 2$.
	Moreover,
	it follows from $Z_{\ovone} f_{m, n} = 0$ that
	\begin{equation}
		Z_{\ovone} u_{m, n}
		= \frac{1}{\sqrt{2 \kappa}} \frac{2 \pi \sqrt{- 1}}{k L} u_{m, n}.
	\end{equation}
	Hence
	\begin{equation}
		\abs*{Z_{\ovone} u_{m, n}}^{2}
		= \frac{2 \pi^{2}}{k^{2} L^{2} \kappa} \abs*{u_{m, n}}^{2}
		\leq \frac{2 \pi^{2}}{k^{2} L^{2} \min \kappa} \abs*{u_{m, n}}^{2},
	\end{equation}
	and so
	\begin{equation}
		\norm{Z_{\ovone} u_{m, n}}_{L^{2}}^{2}
		\leq \frac{2 \pi^{2}}{k^{2} L^{2} \min \kappa} \norm{u_{m, n}}_{L^{2}}^{2}.
	\end{equation}
	\cref{eq:variational-characterization}
	gives the inequality \cref{eq:upper-bound-of-Kohn-Laplacian-on-k-fold-cover}.
	The latter statement follows from the former one and \cref{prop:criterion-of-embeddability}.
\end{proof}

\section{Example 2: Lima\c{c}on}
\label{section:Limacon}

Consider the polar curve given by
\begin{equation}
\label{eq:Limacon}
	r(\theta) = a + b \cos \theta \qquad (0 \leq \theta \leq 2 \pi),
\end{equation}
where $a$ and $b$ are real numbers satisfying $0 < a < b$.
This polar curve is known as \emph{Lima\c{c}on} or \emph{Pascal's Snail}.
Note that the curve \cref{eq:Limacon} is not simple.
Indeed,
take $\theta_{0} \in \opop{\pi / 2}{\pi}$ such that $\cos \theta_{0} = - a / b$
; the existence of $\theta_{0}$ follows from the assumption that $0 < a < b$.
Then $2 \pi - \theta_{0} \in \opop{\pi}{3 \pi / 2}$ and $\cos (2 \pi - \theta_{0}) = - a / b$.
In particular,
the curve \cref{eq:Limacon} passes through the origin twice,
at $\theta = \theta_{0}$ and $\theta = 2 \pi - \theta_{0}$.
The curvature of the curve \cref{eq:Limacon} is given by
\begin{equation}
	\kappa(\theta)
	= \frac{a^{2} + 2 b^{2} + 3 a b \cos \theta}{\rbra*{a^{2} + b^{2} + 2 a b \cos \theta}^{3 / 2}}
	\geq \frac{(2 b - a) (b - a)}{\rbra*{a^{2} + b^{2} + 2 a b \cos \theta}^{3 / 2}}
	> 0.
\end{equation}

Let $s$ be the arc-length function of the curve \cref{eq:Limacon}.
Then we have
\begin{equation}
	\dv{s}{\theta} = \sqrt{a^{2} + b^{2} + 2 a b \cos \theta}.
\end{equation}
In particular,
the length $L$ of the curve \cref{eq:Limacon} is given by
\begin{align}
	L
	= \int_{0}^{2 \pi} \sqrt{a^{2} + b^{2} + 2 a b \cos \theta} \, d \theta
	&= \int_{0}^{2 \pi} \sqrt{(a + b)^{2} - 4 a b \sin^{2} (\theta / 2)} \, d \theta \\
	&= 2 (a + b) \int_{0}^{\pi} \sqrt{1 - \frac{4 a b}{(a + b)^{2}} \sin^{2} \theta} \, d \theta \\
	&= 4 (a + b) \int_{0}^{\pi / 2} \sqrt{1 - \frac{4 a b}{(a + b)^{2}} \sin^{2} \theta} \, d \theta \\
	&= 4 (a + b) E(2 \sqrt{a b} / (a + b)),
\end{align}
where $E(k)$ is the complete elliptic integral of the second kind.

Let $\theta(s)$ be the inverse function of $s(\theta)$
and consider the curve
\begin{equation}
	\gamma(s)
	= (\xi(s), \eta(s))
	\coloneqq ((a + b \cos \theta(s)) \cos \theta(s), (a + b \cos \theta(s)) \sin \theta(s)),
\end{equation}
which is parametrized by arc-length $s$.
Note that the curvature of $\gamma$ is given by $\kappa(\theta(s))$.
Consider the pseudo-Hermitian manifold $(M, T^{1, 0} M, \theta)$ with respect to $\gamma$.
Note that $(M, T^{1, 0} M)$ is non-embeddable
since $\gamma$ is not simple.

\begin{theorem}
	If $b \geq 2 a$,
	then the Tanaka-Webster scalar curvature $\Scal$ is positive,
	and the CR Paneitz operator $P$ has infinitely many negative eigenvalues.
\end{theorem}

\begin{proof}
	We first compute derivatives of $\kappa$:
	\begin{align}
		\frac{d \kappa(\theta(s))}{d s}
		&= \kappa'(\theta(s)) \cdot \theta'(s)
		= 3 a b^{2} \sin \theta(s) (a \cos \theta(s) + b)
			(a^{2} + b^{2} + 2 a b \cos \theta(s))^{- 3}, \\
		\frac{d^{2} \kappa(\theta(s))}{d s^{2}}
		&= 3 a b^{2} \big[ 5 a b^{2} - a^{3} + (b^{3} + 5 a^{2} b) \cos \theta(s)
			+ (- 2 a b^{2} + 2 a^{3}) \cos^{2} \theta(s) \\
		&\quad - 2 a^{2} b \cos^{3} \theta(s) \big]
			(a^{2} + b^{2} + 2 a b \cos \theta(s))^{- 9 / 2}.
	\end{align}
	Now we consider
	\begin{align}
		g(s)
		&\coloneqq (a^{2} + b^{2} + 2 a b \cos \theta(s))^{9 / 2}
		\rbra*{\kappa(\theta(s))^{3} - \frac{d^{2} \kappa(\theta(s))}{d s^{2}}} \\
		&= 8 b^{6} - 3 a^{2} b^{4} + 9 a^{4} b^{2} + a^{6}
			+ (33 a b^{5} + 21 a^{3} b^{3} + 9 a^{5} b) \cos \theta(s) \\
		&\quad + (60 a^{2} b^{4} + 21 a^{4} b^{2}) \cos^{2} \theta(s)
			+ 33 a^{3} b^{3} \cos^{3} \theta(s).
	\end{align}
	It follows from $b \geq 2 a > 0$ that
	$33 a^{3} b^{3} \cos^{3} \theta(s) \geq - 17 a^{2} b^{4} \cos^{2} \theta(s)$,
	and so
	\begin{align}
		g(s)
		&\geq 8 b^{6} - 3 a^{2} b^{4} + 9 a^{4} b^{2} + a^{6}
			+ (33 a b^{5} + 21 a^{3} b^{3} + 9 a^{5} b) \cos \theta(s) \\
		&\quad + (43 a^{2} b^{4} + 21 a^{4} b^{2}) \cos^{2} \theta(s).
	\end{align}
	The right hand side is a quadratic polynomial in $\cos \theta(s)$.
	Set
	\begin{equation}
		Q(x)
		\coloneqq 8 b^{6} - 3 a^{2} b^{4} + 9 a^{4} b^{2} + a^{6}
			+ (33 a b^{5} + 21 a^{3} b^{3} + 9 a^{5} b) x
			+ (43 a^{2} b^{4} + 21 a^{4} b^{2}) x^{2}.
	\end{equation}
	Its discriminant is given by
	\begin{align}
		&(33 a b^{5} + 21 a^{3} b^{3} + 9 a^{5} b)^{2}
		- 4 (43 a^{2} b^{4} + 21 a^{4} b^{2}) (8 b^{6} - 3 a^{2} b^{4} + 9 a^{4} b^{2} + a^{6}) \\
		&= - a^{10} b^{2} (287 t^{4} - 1230 t^{3} + 261 t^{2} + 550 t + 3),
	\end{align}
	where $t = (b / a)^{2} \geq 4$.
	Consider the function
	\begin{equation}
		h(t)
		\coloneqq 287 t^{4} - 1230 t^{3} + 261 t^{2} + 550 t + 3.
	\end{equation}
	Then
	\begin{align}
		h'(t)
		= 1148 t^{3} - 3690 t^{2} + 522 t + 550
		&\geq 4592 t^{2} - 3690 t^{2} + 522 t + 550 \\
		&= 902 t^{2} + 522 t + 550 \\
		&\geq 0
	\end{align}
	for $t \geq 4$.
	Thus we have $h(t) \geq h(4) = 1131 > 0$.
	This implies that the discriminant of $Q(x)$ is negative,
	and the polynomial $Q(x)$ has no real roots.
	It follows from this and $43 a^{2} b^{4} + 21 a^{4} b^{2} > 0$ that
	$Q(x)$ is positive on $\bbR$.
	Hence
	\begin{equation}
		g(s)
		\geq Q(\cos \theta(s)) > 0.
	\end{equation}
	Therefore we have
	\begin{align}
		\Scal
		&= \frac{1}{2 \kappa(\theta(s))^{2}} \rbra*{\kappa(\theta(s))^{3} - \frac{d^{2} \kappa(\theta(s))}{d s^{2}}}
			+ \frac{1}{2 \kappa(\theta(s))^{3}} \rbra*{\frac{d \kappa(\theta(s))}{d s}}^{2} \\
		&\geq \frac{g(s)}{2 \kappa(\theta(s))^{2} (a^{2} + b^{2} + 2 a b \cos \theta(s))^{9 / 2}}
		> 0.
	\end{align}
	The latter statement follows from \cref{thm:infinitely-many-negative-eigenvalues-for-non-simple-curve}
	and the fact that $\gamma$ is not simple.
\end{proof}

\begin{corollary}
	Let $\whxg$ be a sufficiently $C^{4}$-small perturbation of $\gamma$ with $b \geq 2 a$.
	Then the CR Paneitz operator on the pseudo-Hermitian manifold
	$(\whM, T^{1, 0} \whM, \whxth)$ associated with $\whxg$
	has infinitely many negative eigenvalues.
\end{corollary}

\begin{proof}
	It follows from \cref{prop:perturbation-of-non-simple} that
	$\whxg$ satisfies the assumptions in \cref{thm:infinitely-many-negative-eigenvalues-for-non-simple-curve}.
	Hence the CR Paneitz operator on $(\whM, T^{1, 0} \whM, \whxth)$
	has infinitely many negative eigenvalues.
\end{proof}

As we noted,
$(M, T^{1, 0} M)$ is non-embeddable since $\gamma$ is not simple.
We can show that $0$ is an accumulation point of the spectrum of the Kohn Laplacian,
which yields another proof of the non-embeddability.

\begin{theorem}
	The smallest positive eigenvalue $\lambda_{1}(\Box_{b}|_{H_{m, 0}})$ of $\Box_{b}|_{H_{m, 0}}$ satisfies
	\begin{equation}
	\label{eq:upper-bound-of-Kohn-Laplacian-on-Limacon}
		\lambda_{1}(\Box_{b}|_{H_{m, 0}})
		\leq \frac{(a + b)^{2}}{2 (2 b - a) (b - a)^{2}}.
	\end{equation}
	for any negative $m$.
	In particular,
	$0$ is an accumulation point of $\Spec \Box_{b}$.
\end{theorem}

\begin{proof}
	Fix a negative integer $m$.
	To simplify notation,
	set
	\begin{gather}
		f(s, x, y)
		\coloneqq f_{m, 0}(s, x, y)
		= e^{m (a \cos \theta(s) + b \cos^{2} \theta(s)) + \sqrt{- 1} m x} \in \calO \cap H_{m, 0}, \\
		u(s, x, y)
		\coloneqq f(s, x, y) \sin \theta(s) \in H_{m, 0}.
	\end{gather}
	Then
	\begin{align}
		\iproduct{f}{u}_{L^{2}}
		&= 4 \pi^{2} \int_{0}^{L} e^{2 m (a \cos \theta(s) + b \cos^{2} \theta(s))}
			\sin \theta(s) \kappa(\theta(s)) \, d s \\
		&= - 4 \pi^{2} \int_{0}^{2 \pi} e^{2 m (a \cos \theta + b \cos^{2} \theta)}
			\frac{a^{2} + 2 b^{2} + 3 a b \cos \theta}{a^{2} + b^{2} + 2 a b \cos \theta} \, d (\cos \theta) \\
		&= 0.
	\end{align}
	Hence it suffices to show that
	\begin{equation}
		\norm{Z_{\ovone} u}_{L^{2}}^{2}
		\leq \frac{(a + b)^{2}}{2 (2 b - a) (b - a)^{2}} \norm{u}_{L^{2}}^{2}
	\end{equation}
	by \cref{eq:variational-characterization}.
	Since $a^{2} + 2 b^{2} + 3 a b \cos \theta \geq (2 b - a) (b - a)$
	and $a^{2} + b^{2} + 2 a b \cos \theta \leq (a + b)^{2}$,
	\begin{align}
		\norm{u}_{L^{2}}^{2}
		&= 4 \pi^{2} \int_{0}^{L} e^{2 m (a \cos \theta(s) + b \cos^{2} \theta(s))}
			\sin^{2} \theta(s) \kappa(\theta(s)) \, d s \\
		&= 4 \pi^{2} \int_{0}^{2 \pi} \frac{a^{2} + 2 b^{2} + 3 a b \cos \theta}{a^{2} + b^{2} + 2 a b \cos \theta}
			e^{2 m (a \cos \theta + b \cos^{2} \theta)} \sin^{2} \theta \, d \theta \\
		&\geq \frac{4 \pi^{2} (2 b - a) (b - a)}{(a + b)^{2}}
			\int_{0}^{2 \pi} e^{2 m (a \cos \theta + b \cos^{2} \theta)} \sin^{2} \theta \, d \theta.
	\end{align}
	On the other hand,
	it follows from $Z_{\ovone} f = 0$ that
	\begin{equation}
		Z_{\ovone} u
		= f (Z_{\ovone} \sin \theta(s))
		= \frac{\theta'(s)}{\sqrt{2 \kappa(\theta(s))}} f \cos \theta(s),
	\end{equation}
	and so
	\begin{align}
		\norm{Z_{\ovone} u}_{L^{2}}^{2}
		&= 4 \pi^{2} \int_{0}^{L} e^{2 m (a \cos \theta(s) + b \cos^{2} \theta(s))}
			\frac{(\theta'(s))^{2}}{2 \kappa(\theta(s))} \cos^{2} \theta(s) \kappa(\theta(s)) \, d s \\
		&= 2 \pi^{2} \int_{0}^{2 \pi} \frac{1}{\sqrt{a^{2} + b^{2} + 2 a b \cos \theta}}
			e^{2 m (a \cos \theta + b \cos^{2} \theta)} \cos^{2} \theta \, d \theta \\
		&\leq \frac{2 \pi^{2}}{b - a} \int_{0}^{2 \pi}
			e^{2 m (a \cos \theta + b \cos^{2} \theta)} \cos^{2} \theta \, d \theta;
	\end{align}
	in the last inequality,
	we use the fact that $a^{2} + b^{2} + 2 a b \cos \theta \geq (b - a)^{2}$.
	Thus it is enough to prove that
	\begin{equation}
	\label{eq:estimate-of-integral-for-Limacon}
		\int_{0}^{2 \pi} e^{2 m (a \cos \theta + b \cos^{2} \theta)} \cos^{2} \theta \, d \theta
		\leq \int_{0}^{2 \pi} e^{2 m (a \cos \theta + b \cos^{2} \theta)} \sin^{2} \theta \, d \theta.
	\end{equation}
	To this end,
	we consider
	\begin{align}
		\int_{0}^{2 \pi} e^{2 m (a \cos \theta + b \cos^{2} \theta)}
			(\cos^{2} \theta - \sin^{2} \theta) \, d \theta
		&= \int_{0}^{2 \pi} e^{2 m (a \cos \theta + b \cos^{2} \theta)} \cos 2 \theta \, d \theta \\
		&= 2 \int_{0}^{\pi} e^{2 m (a \cos \theta + b \cos^{2} \theta)} \cos 2 \theta \, d \theta
	\end{align}
	and we show that this integral is non-positive.
	We can write
	\begin{align}
		\int_{0}^{\pi} e^{2 m (a \cos \theta + b \cos^{2} \theta)} \cos 2 \theta \, d \theta
		&= \rbra*{\int_{0}^{\pi / 2} + \int_{\pi / 2}^{\pi}}
			e^{2 m (a \cos \theta + b \cos^{2} \theta)} \cos 2 \theta \, d \theta \\
		&\eqqcolon I_{1} + I_{2}.
	\end{align}
	The change of variables $\wtxth = \pi / 2 - \theta$ implies
	\begin{align}
		I_{1}
		&= \int_{0}^{\pi / 4} e^{2 m (a \cos \theta + b \cos^{2} \theta)} \cos 2 \theta \, d \theta
			+ \int_{\pi / 4}^{\pi / 2} e^{2 m (a \cos \theta + b \cos^{2} \theta)} \cos 2 \theta \, d \theta \\
		&= \int_{0}^{\pi / 4} \rbra*{e^{2 m (a \cos \theta + b \cos^{2} \theta)}
			- e^{2 m (a \sin \theta + b \sin^{2} \theta)}} \cos 2 \theta \, d \theta.
	\end{align}
	Since $\cos \theta \geq \sin \theta \geq 0$ on $\clcl{0}{\pi / 4}$,
	we have
	\begin{equation}
		e^{2 m (a \cos \theta + b \cos^{2} \theta)}
		\leq e^{2 m (a \sin \theta + b \sin^{2} \theta)}
	\end{equation}
	on $\clcl{0}{\pi / 4}$ for negative $m$.
	This and $\cos 2 \theta \geq 0$ on $\clcl{0}{\pi / 4}$ show $I_{1} \leq 0$.
	Similarly,
	\begin{align}
		I_{2}
		&= \int_{\pi / 2}^{\pi} e^{2 m (a \cos \theta + b \cos^{2} \theta)} \cos 2 \theta \, d \theta \\
		&= \int_{\pi / 2}^{3 \pi / 4} e^{2 m (a \cos \theta + b \cos^{2} \theta)} \cos 2 \theta \, d \theta
			+ \int_{3 \pi / 4}^{\pi} e^{2 m (a \cos \theta + b \cos^{2} \theta)} \cos 2 \theta \, d \theta \\
		&= \int_{\pi / 2}^{3 \pi / 4} \rbra*{e^{2 m (a \cos \theta + b \cos^{2} \theta)} 
			- e^{2 m (- a \sin \theta + b \sin^{2} \theta)}} \cos 2 \theta \, d \theta,
	\end{align}
	where we used the change of variables $\wtxth = 3 \pi / 2 - \theta$ to rewrite the second integral.
	Since $\sin \theta \geq - \cos \theta \geq 0$ and $\sin \theta - \cos \theta \geq 1$
	on $\clcl{\pi / 2}{3 \pi / 4}$,
	we have
	\begin{align}
		(a \cos \theta + b \cos^{2} \theta) - (- a \sin \theta + b \sin^{2} \theta)
		&= (a - b (\sin \theta - \cos \theta)) (\sin \theta + \cos \theta) \\
		&\leq 0
	\end{align}
	on $\clcl{\pi / 2}{3 \pi / 4}$.
	This yields that
	\begin{equation}
		e^{2 m (a \cos \theta + b \cos^{2} \theta)}
		\geq e^{2 m (- a \sin \theta + b \sin^{2} \theta)}
	\end{equation}
	on $\clcl{\pi / 2}{3 \pi / 4}$ for negative $m$.
	This and $\cos 2 \theta \leq 0$ on $\clcl{\pi / 2}{3 \pi / 4}$ imply $I_{2} \leq 0$.
	Combining these yields $I_{1} + I_{2} \leq 0$,
	which proves \cref{eq:estimate-of-integral-for-Limacon}.
	The latter statement follows from the former one and \cref{prop:criterion-of-embeddability}.
\end{proof}

\section{Example 3: Hypotrochoid}
\label{section:Hypotrochoid}

Consider the curve given by
\begin{equation}
\label{eq:hypotrochoid}
	(x(\theta), y(\theta))
	= (2 a \cos \theta + b \cos 2 \theta, - 2 a \sin \theta + b \sin 2 \theta)
	\qquad (0 \leq \theta \leq 2 \pi),
\end{equation}
where $a$ and $b$ are real numbers satisfying $0 < a < b$.
This curve is an example of a \emph{hypotrochoid}.
Note that the curve \cref{eq:hypotrochoid} is not simple.
Indeed,
take $\theta_{0} \in \opop{0}{\pi / 2}$ such that $\cos \theta_{0} = a / b$;
the existence of $\theta_{0}$ follows from the assumption that $0 < a < b$.
Then $2 \pi - \theta_{0} \in \opop{3 \pi / 2}{2 \pi}$ and
\begin{equation}
	x(\theta_{0}) = x(2 \pi - \theta_{0}),
	\qquad
	y(\theta_{0}) = y(2 \pi - \theta_{0}) = 0.
\end{equation}
The curvature of the curve \cref{eq:hypotrochoid} is given by
\begin{equation}
	\kappa(\theta)
	= \frac{- a^{2} + 2 b^{2} - a b \cos 3 \theta}{2 \rbra*{a^{2} + b^{2} - 2 a b \cos 3 \theta}^{3 / 2}}
	\geq \frac{(2 b + a) (b - a)}{2 \rbra*{a^{2} + b^{2} - 2 a b \cos 3 \theta}^{3 / 2}}
	> 0.
\end{equation}

Let $s$ be the arc-length function of the curve \cref{eq:hypotrochoid}.
Then we have
\begin{equation}
	\dv{s}{\theta} = 2 \sqrt{a^{2} + b^{2} - 2 a b \cos 3 \theta}.
\end{equation}
In particular,
the length $L$ of the curve \cref{eq:hypotrochoid} is given by
\begin{align}
	L
	= 2 \int_{0}^{2 \pi} \sqrt{a^{2} + b^{2} - 2 a b \cos 3 \theta} \, d \theta
	&= \frac{4}{3} \int_{0}^{3 \pi} \sqrt{a^{2} + b^{2} - 2 a b \cos 2 \theta} \, d \theta \\
	&= \frac{4}{3} \int_{0}^{3 \pi} \sqrt{(a + b)^{2} - 4 a b \cos^{2} \theta} \, d \theta \\
	&= 8 (a + b) \int_{0}^{\pi / 2} \sqrt{1 - \frac{4 a b}{(a + b)^{2}} \cos^{2} \theta} \, d \theta \\
	&= 8 (a + b) \int_{0}^{\pi / 2} \sqrt{1 - \frac{4 a b}{(a + b)^{2}} \sin^{2} \theta} \, d \theta \\
	&= 8 (a + b) E(2 \sqrt{a b} / (a + b)).
\end{align}

Let $\theta(s)$ be the inverse function of $s(\theta)$
and consider the curve
\begin{equation}
	\gamma(s)
	= (\xi(s), \eta(s))
	\coloneqq (2 a \cos \theta(s) + b \cos 2 \theta(s), - 2 a \sin \theta(s) + b \sin 2 \theta(s)),
\end{equation}
which is parametrized by arc-length $s$.
Note that the curvature of $\gamma$ is given by $\kappa(\theta(s))$.
Consider the pseudo-Hermitian manifold $(M, T^{1, 0} M, \theta)$ with respect to $\gamma$.
Note that $(M, T^{1, 0} M)$ is non-embeddable
since $\gamma$ is not simple.

\begin{theorem}
	If $b \geq 2 \sqrt{11} a$,
	then the Tanaka-Webster scalar curvature $\Scal$ is positive,
	and the CR Paneitz operator $P$ has infinitely many negative eigenvalues.
\end{theorem}

\begin{proof}
	We first compute derivatives of $\kappa$:
	\begin{align}
		\frac{d \kappa(\theta(s))}{d s}
		&= \kappa'(\theta(s)) \cdot \theta'(s)
		= \frac{3 a b \sin 3 \theta(s) (4a^{2} - 5 b^{2} + a b \cos 3 \theta(s))}
			{4 (a^{2} + b^{2} - 2 a b \cos 3 \theta(s))^{3}}, \\
		\frac{d^{2} \kappa(\theta(s))}{d s^{2}}
		&= \frac{1}{8} \big[ 261 a^{2} b^{4} - 225 a^{4} b^{2}
			+ (- 45 a b^{5} - 45 a^{3} b^{3} + 36 a^{5} b) \cos 3 \theta(s) \\
		&\quad + (- 162 a^{2} b^{4} + 162 a^{4} b^{2}) \cos^{2} 3 \theta(s)
			+ 18 a^{3} b^{3} \cos^{3} 3 \theta(s) \big] \\
		&\quad \times (a^{2} + b^{2} - 2 a b \cos 3 \theta(s))^{- 9 / 2}.
	\end{align}
	Now we consider
	\begin{align}
		g(s)
		&\coloneqq 8 (a^{2} + b^{2} - 2 a b \cos 3 \theta(s))^{9 / 2}
			\rbra*{\kappa(\theta(s))^{3} - \frac{d^{2} \kappa(\theta(s))}{d s^{2}}} \\
		&= 8 b^{6} - 273 a^{2} b^{4} + 231 a^{4} b^{2} - a^{6}
			+ (33 a b^{5} + 57 a^{3} b^{3} - 39 a^{5} b) \cos 3 \theta(s) \\
		&\quad + (168 a^{2} b^{4} - 165 a^{4} b^{2}) \cos^{2} 3 \theta(s) - 19 a^{3} b^{3} \cos^{3} 3 \theta(s).
	\end{align}
	It follows from $b \geq 2 \sqrt{11} a$ that
	$- 19 a^{3} b^{3} \cos^{3} 3 \theta(s) \geq - 4 a^{2} b^{4} \cos^{2} 3 \theta(s)$,
	and so
	\begin{align}
		g(s)
		&\geq 8 b^{6} - 273 a^{2} b^{4} + 231 a^{4} b^{2} - a^{6}
			+ (33 a b^{5} + 57 a^{3} b^{3} - 39 a^{5} b) \cos 3 \theta(s) \\
		&\quad + (164 a^{2} b^{4} - 165 a^{4} b^{2}) \cos^{2} 3 \theta(s).
	\end{align}
	The right hand side is a quadratic polynomial in $\cos 3 \theta(s)$.
	Set
	\begin{align}
		Q(x)
		&\coloneqq 8 b^{6} - 273 a^{2} b^{4} + 231 a^{4} b^{2} - a^{6}
			+ (33 a b^{5} + 57 a^{3} b^{3} - 39 a^{5} b) x \\
		&\quad + (164 a^{2} b^{4} - 165 a^{4} b^{2}) x^{2}.
	\end{align}
	Its discriminant is given by
	\begin{align}
		&(33 a b^{5} + 57 a^{3} b^{3} - 39 a^{5} b)^{2}
		- 4 (164 a^{2} b^{4} - 165 a^{4} b^{2}) (8 b^{6} - 273 a^{2} b^{4} + 231 a^{4} b^{2} - a^{6}) \\
		&= - a^{10} b^{2} (4159 t^{4} - 188130 t^{3} + 331041 t^{2} -148670 t - 861),
	\end{align}
	where $t = (b / a)^{2} \geq 44$.
	Consider the function
	\begin{equation}
		h(t)
		\coloneqq 4159 t^{4} - 188130 t^{3} + 331041 t^{2} -148670 t - 861.
	\end{equation}
	Then
	\begin{align}
		h'(t)
		&= 16636 t^{3} - 564390 t^{2} + 662082 t - 148670 \\
		&\geq 731984 t^{2} - 564390 t^{2} + 662082 - 148670 \\
		&= 167594 t^{2} + 513412 \\
		&\geq 0
	\end{align}
	for $t \geq 44$.
	Thus we have $h(t) \geq h(44) = 197018379 > 0$.
	This implies that the discriminant of $Q(x)$ is negative,
	and the polynomial $Q(x)$ has no real roots.
	It follows from this and $164 a^{2} b^{4} - 165 a^{4} b^{2} > 0$ that
	$Q(x)$ is positive on $\bbR$.
	Hence
	\begin{equation}
		g(s)
		\geq Q(\cos 3 \theta(s)) > 0.
	\end{equation}
	Therefore we have
	\begin{align}
		\Scal
		&= \frac{1}{2 \kappa(\theta(s))^{2}} \rbra*{\kappa(\theta(s))^{3} - \frac{d^{2} \kappa(\theta(s))}{d s^{2}}}
			+ \frac{1}{2 \kappa(\theta(s))^{3}} \rbra*{\frac{d \kappa(\theta(s))}{d s}}^{2} \\
		&\geq \frac{g(s)}{16 \kappa(\theta(s))^{2} (a^{2} + b^{2} - 2 a b \cos 3 \theta(s))^{9 / 2}}
		> 0.
	\end{align}
	The latter statement follows from \cref{thm:infinitely-many-negative-eigenvalues-for-non-simple-curve}
	and the fact that $\gamma$ is not simple.
\end{proof}

\begin{corollary}
	Let $\whxg$ be a sufficiently $C^{4}$-small perturbation of $\gamma$ with $b \geq 2 \sqrt{11} a$.
	Then the CR Paneitz operator on the pseudo-Hermitian manifold
	$(\whM, T^{1, 0} \whM, \whxth)$ associated with $\whxg$
	has infinitely many negative eigenvalues.
\end{corollary}

\begin{proof}
	It follows from \cref{prop:perturbation-of-non-simple} that
	$\whxg$ satisfies the assumptions in \cref{thm:infinitely-many-negative-eigenvalues-for-non-simple-curve}.
	Hence the CR Paneitz operator on $(\whM, T^{1, 0} \whM, \whxth)$
	has infinitely many negative eigenvalues.
\end{proof}

As we noted,
$(M, T^{1, 0} M)$ is non-embeddable since $\gamma$ is not simple.
We can show that $0$ is an accumulation point of the spectrum of the Kohn Laplacian,
which yields another proof of the non-embeddability.

\begin{theorem}
	The smallest positive eigenvalue $\lambda_{1}(\Box_{b}|_{H_{m, 0}})$ of $\Box_{b}|_{H_{m, 0}}$ satisfies
	\begin{equation}
	\label{eq:upper-bound-of-Kohn-Laplacian-on-hypotrochoid}
		\lambda_{1}(\Box_{b}|_{H_{m, 0}})
		\leq \frac{(a + b)^{2}}{4 (2 b + a) (b - a)^{2}}.
	\end{equation}
	for any negative integer $m$.
	In particular,
	$0$ is an accumulation point of $\Spec \Box_{b}$.
\end{theorem}

\begin{proof}
	Fix a negative integer $m$.
	To simplify notation,
	set
	\begin{gather}
		f(s, x, y)
		\coloneqq f_{m, 0}(s, x, y)
		= e^{m (2 a \cos \theta(s) + b \cos 2 \theta(s)) + \sqrt{- 1} m x} \in \calO \cap H_{m, 0}, \\
		u(s, x, y)
		\coloneqq f(s, x, y) \sin \theta(s) \in H_{m, 0}.
	\end{gather}
	Then
	\begin{align}
		\iproduct{f}{u}_{L^{2}}
		&= 4 \pi^{2} \int_{0}^{L} e^{2 m (2 a \cos \theta(s) + b \cos 2 \theta(s))}
			\sin \theta(s) \kappa(\theta(s)) \, d s \\
		&= 4 \pi^{2} \int_{0}^{2 \pi}
			\frac{- a^{2} + 2 b^{2} - a b \cos 3 \theta}{a^{2} + b^{2} - 2 a b \cos 3 \theta}
			e^{2 m (2 a \cos \theta + b \cos 2 \theta)} \sin \theta \, d \theta \\
		&= 4 \pi^{2} \int_{0}^{\pi}
			\frac{- a^{2} + 2 b^{2} - a b \cos 3 \theta}{a^{2} + b^{2} - 2 a b \cos 3 \theta}
			e^{2 m (2 a \cos \theta + b \cos 2 \theta)} \sin \theta \, d \theta \\
		&\quad - 4 \pi^{2} \int_{0}^{\pi}
			\frac{- a^{2} + 2 b^{2} - a b \cos 3 \theta}{a^{2} + b^{2} - 2 a b \cos 3 \theta}
			e^{2 m (2 a \cos \theta + b \cos 2 \theta)} \sin \theta \, d \theta \\
		&= 0,
	\end{align}
	where we used the change of variables $\wtxth = 2 \pi - \theta$ to rewrite the second integral.
	Hence it suffices to show that
	\begin{equation}
		\norm{Z_{\ovone} u}_{L^{2}}^{2}
		\leq \frac{(a + b)^{2}}{4 (2 b + a) (b - a)^{2}} \norm{u}_{L^{2}}^{2}
	\end{equation}
	by \cref{eq:variational-characterization}.
	The definition of $u$ yields
	\begin{align}
		\norm{u}_{L^{2}}^{2}
		&= 4 \pi^{2} \int_{0}^{L} e^{2 m (2 a \cos \theta(s) + b \cos 2 \theta(s))}
			\sin^{2} \theta(s) \kappa(\theta(s)) \, d s \\
		&= 4 \pi^{2} \int_{0}^{2 \pi}
			\frac{- a^{2} + 2 b^{2} - a b \cos 3 \theta}{a^{2} + b^{2} - 2 a b \cos 3 \theta}
			e^{2 m (2 a \cos \theta + b \cos 2 \theta)} \sin^{2} \theta \, d \theta \\
		&\geq \frac{4 \pi^{2} (2 b + a) (b - a)}{(a + b)^{2}}
			\int_{0}^{2 \pi} e^{2 m (2 a \cos \theta + b \cos 2 \theta)} \sin^{2} \theta \, d \theta.
	\end{align}
	On the other hand,
	it follows from $Z_{\ovone} f = 0$ that
	\begin{equation}
		Z_{\ovone} u
		= f (Z_{\ovone} \sin \theta(s))
		= \frac{\theta'(s)}{\sqrt{2 \kappa(\theta(s))}} f \cos \theta(s),
	\end{equation}
	and so
	\begin{align}
		\norm{Z_{\ovone} u}_{L^{2}}^{2}
		&= 4 \pi^{2} \int_{0}^{L} e^{2 m (2 a \cos \theta(s) + b \cos 2 \theta(s))}
			\frac{(\theta'(s))^{2}}{2 \kappa(\theta(s))} \cos^{2} \theta(s) \kappa(\theta(s)) \, d s \\
		&= 2 \pi^{2} \int_{0}^{2 \pi} \frac{1}{2 \sqrt{a^{2} + b^{2} - 2 a b \cos 3 \theta}}
			e^{2 m (2 a \cos \theta + b \cos 2 \theta)} \cos^{2} \theta \, d \theta \\
		&\leq \frac{\pi^{2}}{b - a} \int_{0}^{2 \pi}
			e^{2 m (2 a \cos \theta + b \cos 2 \theta)} \cos^{2} \theta \, d \theta.
	\end{align}
	Thus it is enough to prove
	\begin{equation}
	\label{eq:estimate-of-integral-for-hypotrochoid}
		\int_{0}^{2 \pi} e^{2 m (2 a \cos \theta + b \cos 2 \theta)} \cos^{2} \theta \, d \theta
		\leq \int_{0}^{2 \pi} e^{2 m (2 a \cos \theta + b \cos 2 \theta)} \sin^{2} \theta \, d \theta.
	\end{equation}
	To this end,
	we consider
	\begin{align}
		\int_{0}^{2 \pi} e^{2 m (2 a \cos \theta + b \cos 2 \theta)} (\cos^{2} \theta - \sin^{2} \theta) \, d \theta
		&= \int_{0}^{2 \pi} e^{2 m (2 a \cos \theta + b \cos 2 \theta)} \cos 2 \theta \, d \theta \\
		&= 2 \int_{0}^{\pi} e^{2 m (2 a \cos \theta + b \cos 2 \theta)} \cos 2 \theta \, d \theta,
	\end{align}
	and we show that this integral is non-positive.
	We can write
	\begin{align}
		\int_{0}^{\pi} e^{2 m (2 a \cos \theta + b \cos 2 \theta)} \cos 2 \theta \, d \theta
		&= \rbra*{\int_{0}^{\pi / 2} + \int_{\pi / 2}^{\pi}}
			e^{2 m (2 a \cos \theta + b \cos 2 \theta)} \cos 2 \theta \, d \theta \\
		&\eqqcolon I_{1} + I_{2}.
	\end{align}
	The change of variables $\wtxth = \pi / 2 - \theta$ implies
	\begin{align}
		I_{1}
		&= \int_{0}^{\pi / 4} e^{2 m (2 a \cos \theta + b \cos 2 \theta)} \cos 2 \theta \, d \theta
			+ \int_{\pi / 4}^{\pi / 2} e^{2 m (2 a \cos \theta + b \cos 2 \theta)} \cos 2 \theta \, d \theta \\
		&= \int_{0}^{\pi / 4} \rbra*{e^{2 m (2 a \cos \theta + b \cos 2 \theta)}
			- e^{2 m (2 a \sin \theta - b \cos 2 \theta)}} \cos 2 \theta \, d \theta.
	\end{align}
	Since $\cos \theta \geq \sin \theta \geq 0$ and $\cos 2 \theta \geq 0$ on $\clcl{0}{\pi / 4}$,
	we have
	\begin{equation}
		e^{2 m (2 a \cos \theta + b \cos 2 \theta)}
		\leq e^{2 m (2 a \sin \theta - b \cos 2 \theta)}
	\end{equation}
	on $\clcl{0}{\pi / 4}$ for negative $m$.
	This and $\cos 2 \theta \geq 0$ on $\clcl{0}{\pi / 4}$ show $I_{1} \leq 0$.
	Similarly,
	\begin{align}
		I_{2}
		&= \int_{\pi / 2}^{3 \pi / 4} e^{2 m (2 a \cos \theta + b \cos 2 \theta)} \cos 2 \theta \, d \theta
			+ \int_{3 \pi / 4}^{\pi} e^{2 m (2 a \cos \theta + b \cos 2 \theta)} \cos 2 \theta \, d \theta \\
		&= \int_{\pi / 2}^{3 \pi / 4} \rbra*{e^{2 m (2 a \cos \theta + b \cos 2 \theta)}
			- e^{2 m (- 2 a \sin \theta - b \cos 2 \theta)}} \cos 2 \theta \, d \theta,
	\end{align}
	where we used the change of variables $\wtxth = 3 \pi / 2 - \theta$ to rewrite the second integral.
	Since $\sin \theta \geq - \cos \theta \geq 0$
	and $\sin \theta - \cos \theta \geq 1$ on $\clcl{\pi / 2}{3 \pi / 4}$,
	we have
	\begin{align}
		(2 a \cos \theta + b \cos 2 \theta) - (- 2 a \sin \theta - b \cos 2 \theta)
		&= 2 (\cos \theta + \sin \theta) (a - b (\sin \theta - \cos \theta)) \\
		&\leq 0
	\end{align}
	on $\clcl{\pi / 2}{3 \pi / 4}$.
	This yields that
	\begin{equation}
		e^{2 m (2 a \cos \theta + b \cos 2 \theta)}
		\geq e^{2 m (- 2 a \cos \theta - b \cos 2 \theta)}
	\end{equation}
	on $\clcl{\pi / 2}{3 \pi / 4}$ for negative $m$.
	This and $\cos 2 \theta \leq 0$ on $\clcl{\pi / 2}{3 \pi / 4}$ imply $I_{2} \leq 0$.
	Combining these yields $I_{1} + I_{2} \leq 0$,
	which proves \cref{eq:estimate-of-integral-for-hypotrochoid}.
	The latter statement follows from the former one and \cref{prop:criterion-of-embeddability}.
\end{proof}

\section{Example 4: Epitrochoid}
\label{section:Epitrochoid}

Consider the curve given by
\begin{equation}
\label{eq:epitrochoid}
	(x(\theta), y(\theta))
	= (3 a \cos \theta - b \cos 3 \theta, 3 a \sin \theta - b \sin 3 \theta)
	\qquad (0 \leq \theta \leq 2 \pi),
\end{equation}
where $a$ and $b$ are real numbers satisfying $0 < a < b$.
This curve is an example of an \emph{epitrochoid}.
Note that the curve \cref{eq:epitrochoid} is not simple.
Indeed,
take $\theta_{0} \in \opop{0}{\pi / 2}$ such that
\begin{equation}
	\sin \theta_{0}
	= \sqrt{\frac{3 (b - a)}{4 b}};
\end{equation}
the existence of $\theta_{0}$ follows from the assumption that $0 < a < b$.
Then $2 \pi - \theta_{0} \in \opop{3 \pi / 2}{2 \pi}$ and
\begin{equation}
	x(\theta_{0}) = x(2 \pi - \theta_{0}),
	\qquad
	y(\theta_{0}) = y(2 \pi - \theta_{0}) = 0.
\end{equation}
The curvature of the curve \cref{eq:epitrochoid} is given by
\begin{equation}
	\kappa(\theta)
	= \frac{a^{2} + 3 b^{2} - 4 a b \cos 2 \theta}{3 \rbra*{a^{2} + b^{2} - 2 a b \cos 2 \theta}^{3 / 2}}
	\geq \frac{(3 b - a) (b - a)}{3 \rbra*{a^{2} + b^{2} - 2 a b \cos 2 \theta}^{3 / 2}}
	> 0.
\end{equation}

Let $s$ be the arc-length function of the curve \cref{eq:epitrochoid}.
Then we have
\begin{equation}
	\dv{s}{\theta} = 3 \sqrt{a^{2} + b^{2} - 2 a b \cos 2 \theta}.
\end{equation}
In particular,
the length $L$ of the curve \cref{eq:epitrochoid} is given by
\begin{align}
	L
	= 3 \int_{0}^{2 \pi} \sqrt{a^{2} + b^{2} - 2 a b \cos 2 \theta} \, d \theta
	&= 3 \int_{0}^{2 \pi} \sqrt{(a + b)^{2} - 4 a b \cos^{2} \theta} \, d \theta \\
	&= 3 \int_{0}^{2 \pi} \sqrt{(a + b)^{2} - 4 a b \sin^{2} \theta} \, d \theta \\
	&= 12 (a + b) \int_{0}^{\pi / 2} \sqrt{1 - \frac{4 a b}{(a + b)^{2}} \sin^{2} \theta} \, d \theta \\
	&= 12 (a + b) E(2 \sqrt{a b} / (a + b)).
\end{align}

Let $\theta(s)$ be the inverse function of $s(\theta)$
and consider the curve
\begin{equation}
	\gamma(s)
	= (\xi(s), \eta(s))
	\coloneqq (3 a \cos \theta(s) - b \cos 3 \theta(s), 3 a \sin \theta(s) - b \sin 3 \theta(s)),
\end{equation}
which is parametrized by arc-length $s$.
Note that the curvature of $\gamma$ is given by $\kappa(\theta(s))$.
Consider the pseudo-Hermitian manifold $(M, T^{1, 0} M, \theta)$ with respect to $\gamma$.
Note that $(M, T^{1, 0} M)$ is non-embeddable
since $\gamma$ is not simple.

\begin{theorem}
	If $b \geq \sqrt{5} a$,
	then the Tanaka-Webster scalar curvature $\Scal$ is positive,
	and the CR Paneitz operator $P$ has infinitely many negative eigenvalues.
\end{theorem}

\begin{proof}
	We first compute derivatives of $\kappa$:
	\begin{align}
		\frac{d \kappa(\theta(s))}{d s}
		&= \kappa'(\theta(s)) \cdot \theta'(s)
		= \frac{2 a b \sin 2 \theta(s) (a^{2} - 5 b^{2} + 4 a b \cos 2 \theta(s))}
			{9 (a^{2} + b^{2} - 2 a b \cos 2 \theta(s))^{3}}, \\
		\frac{d^{2} \kappa(\theta(s))}{d s^{2}}
		&= \frac{1}{27} \big[ 104 a^{2} b^{4} - 40 a^{4} b^{2}
			+ (- 20 a b^{5} - 80 a^{3} b^{3} + 4 a^{5} b) \cos 2 \theta(s) \\
		&\quad + (- 48 a^{2} b^{4} + 48 a^{4} b^{2}) \cos^{2} 2 \theta(s)
			+ 32 a^{3} b^{3} \cos^{3} 2 \theta(s) \big] \\
		&\quad \times (a^{2} + b^{2} - 2 a b \cos 2 \theta(s))^{- 9 / 2}.
	\end{align}
	Now we consider
	\begin{align}
		g(s)
		&\coloneqq 27 (a^{2} + b^{2} - 2 a b \cos 2 \theta(s))^{9 / 2}
		\rbra*{\kappa(\theta(s))^{3} - \frac{d^{2} \kappa(\theta(s))}{d s^{2}}} \\
		&= 27 b^{6} - 77 a^{2} b^{4} + 49 a^{4} b^{2} + a^{6}
			+ (- 88 a b^{5} + 8 a^{3} b^{3} - 16 a^{5} b) \cos 2 \theta(s) \\
		&\quad + 192 a^{2} b^{4} \cos^{2} 2 \theta(s) - 96 a^{3} b^{3} \cos^{3} 2 \theta(s).
	\end{align}
	If $b \geq \sqrt{5} a$,
	then $- 96 a^{3} b^{3} \cos^{3} 2 \theta(s) \geq - 48 a^{2} b^{4} \cos^{2} 2 \theta(s)$,
	and so
	\begin{align}
		g(s)
		&\geq 27 b^{6} - 77 a^{2} b^{4} + 49 a^{4} b^{2} + a^{6} \\
		&\quad + (- 88 a b^{5} + 8 a^{3} b^{3} - 16 a^{5} b) \cos 2 \theta(s)
			+ 144 a^{2} b^{4} \cos^{2} 2 \theta(s).
	\end{align}
	The right hand side is a quadratic polynomial in $\cos 2 \theta(s)$.
	Set
	\begin{equation}
		Q(x)
		\coloneqq 27 b^{6} - 77 a^{2} b^{4} + 49 a^{4} b^{2} + a^{6}
			+ (- 88 a b^{5} + 8 a^{3} b^{3} - 16 a^{5} b) x 
			+ 144 a^{2} b^{4} x^{2}.
	\end{equation}
	Its discriminant is given by
	\begin{align}
		&(- 88 a b^{5} + 8 a^{3} b^{3} - 16 a^{5} b)^{2}
		- 4 \cdot 144 a^{2} b^{4} (27 b^{6} - 77 a^{2} b^{4} + 49 a^{4} b^{2} + a^{6}) \\
		&= - 64 a^{10} b^{2} (122 t^{4} - 671 t^{3} + 396 t^{2} + 13 t - 4),
	\end{align}
	where $t = (b / a)^{2} \geq 5$.
	Consider the function
	\begin{equation}
		h(t)
		\coloneqq 122 t^{4} - 671 t^{3} + 396 t^{2} + 13 t - 4.
	\end{equation}
	Then
	\begin{align}
		h'(t)
		= 488 t^{3} - 2013 t^{2} + 792 t + 13
		&\geq 2440 t^{2} - 2013 t^{2} + 792 t + 13 \\
		&= 427 t^{2} + 792 t + 13 \\
		&\geq 0
	\end{align}
	for $t \geq 5$.
	Thus we have $h(t) \geq h(5) = 2336 > 0$.
	This implies that the discriminant of $Q(x)$ is negative,
	and the polynomial $Q(x)$ has no real roots.
	It follows from this and $144 a^{2} b^{4} > 0$ that
	$Q(x)$ is positive on $\bbR$.
	Hence
	\begin{equation}
		g(s)
		\geq Q(\cos 2 \theta(s)) > 0.
	\end{equation}
	Therefore we have
	\begin{align}
		\Scal
		&= \frac{1}{2 \kappa(\theta(s))^{2}} \rbra*{\kappa(\theta(s))^{3} - \frac{d^{2} \kappa(\theta(s))}{d s^{2}}}
			+ \frac{1}{2 \kappa(\theta(s))^{3}} \rbra*{\frac{d \kappa(\theta(s))}{d s}}^{2} \\
		&\geq \frac{g(s)}{54 \kappa(\theta(s))^{2} (a^{2} + b^{2} - 2 a b \cos 2 \theta(s))^{9 / 2}}
		> 0.
	\end{align}
	The latter statement follows from \cref{thm:infinitely-many-negative-eigenvalues-for-non-simple-curve}
	and the fact that $\gamma$ is not simple.
\end{proof}

\begin{corollary}
	Let $\whxg$ be a sufficiently $C^{4}$-small perturbation of $\gamma$ with $b \geq 2 \sqrt{11} a$.
	Then the CR Paneitz operator on the pseudo-Hermitian manifold
	$(\whM, T^{1, 0} \whM, \whxth)$ associated with $\whxg$
	has infinitely many negative eigenvalues.
\end{corollary}

\begin{proof}
	It follows from \cref{prop:perturbation-of-non-simple} that
	$\whxg$ satisfies the assumptions in \cref{thm:infinitely-many-negative-eigenvalues-for-non-simple-curve}.
	Hence the CR Paneitz operator on $(\whM, T^{1, 0} \whM, \whxth)$
	has infinitely many negative eigenvalues.
\end{proof}

As we noted,
$(M, T^{1, 0} M)$ is non-embeddable since $\gamma$ is not simple.
We can show that $0$ is an accumulation point of the spectrum of the Kohn Laplacian in the case of $b = 3 a$,
which yields another proof of the non-embeddability.

\begin{theorem}
	If $b = 3 a$,
	then the smallest positive eigenvalue $\lambda_{1}(\Box_{b}|_{H_{m, 0}})$
	of $\Box_{b}|_{H_{m, 0}}$ satisfies
	\begin{equation}
	\label{eq:upper-bound-of-Kohn-Laplacian-on-epitrochoid}
		\lambda_{1}(\Box_{b}|_{H_{m, 0}})
		\leq \frac{1}{12 a}.
	\end{equation}
	for any integer $m$.
	In particular,
	$0$ is an accumulation point of $\Spec \Box_{b}$.
\end{theorem}

\begin{proof}
	Fix an integer $m$.
	To simplify notation,
	set
	\begin{gather}
		f(s, x, y)
		\coloneqq f_{m, 0}(s, x, y)
		= e^{3 a m (\cos \theta(s) - \cos 3 \theta(s)) + \sqrt{- 1} m x} \in \calO \cap H_{m, 0}, \\
		u(s, x, y)
		\coloneqq f(s, x, y) \sin \theta(s) \in H_{m, 0}.
	\end{gather}
	Then
	\begin{align}
		\iproduct{f}{u}_{L^{2}}
		&= 4 \pi^{2} \int_{0}^{L} e^{6 a m (\cos \theta(s) - \cos 3 \theta(s))}
			\sin \theta(s) \kappa(\theta(s)) \, d s \\
		&= 4 \pi^{2} \int_{0}^{2 \pi} \frac{14 - 6 \cos 2 \theta}{5 - 3 \cos 2 \theta}
			e^{6 a m (\cos \theta - \cos 3 \theta)} \sin \theta \, d \theta \\
		&= 4 \pi^{2} \int_{0}^{\pi} \frac{14 - 6 \cos 2 \theta}{5 - 3 \cos 2 \theta}
			e^{6 a m (\cos \theta - \cos 3 \theta)} \sin \theta \, d \theta \\
		&\quad - 4 \pi^{2} \int_{0}^{\pi} \frac{14 - 6 \cos 2 \theta}{5 - 3 \cos 2 \theta}
			e^{6 a m (\cos \theta - \cos 3 \theta)} \sin \theta \, d \theta \\
		&= 0,
	\end{align}
	where we used the change of variables $\wtxth = 2 \pi - \theta$ to rewrite the second integral.
	Hence it suffices to show that
	\begin{equation}
		\norm{Z_{\ovone} u}_{L^{2}}^{2}
		\leq \frac{1}{12 a} \norm{u}_{L^{2}}^{2}
	\end{equation}
	by \cref{eq:variational-characterization}.
	The definition of $u$ yields
	\begin{align}
		\norm{u}_{L^{2}}^{2}
		&= 4 \pi^{2} \int_{0}^{L} e^{6 a m (\cos \theta(s) - \cos 3 \theta(s))}
			\sin^{2} \theta(s) \kappa(\theta(s)) \, d s \\
		&= 4 \pi^{2} \int_{0}^{2 \pi} \frac{14 - 6 \cos 2 \theta}{5 - 3 \cos 2 \theta}
			e^{6 a m (\cos \theta - \cos 3 \theta)} \sin^{2} \theta \, d \theta \\
		&\geq 4 \pi^{2}
			\int_{0}^{2 \pi} e^{6 a m (\cos \theta - \cos 3 \theta)} \sin^{2} \theta \, d \theta.
	\end{align}
	On the other hand,
	it follows from $Z_{\ovone} f = 0$ that
	\begin{equation}
		Z_{\ovone} u
		= f (Z_{\ovone} \sin \theta(s))
		= \frac{\theta'(s)}{\sqrt{2 \kappa(\theta(s))}} f \cos \theta(s),
	\end{equation}
	and so
	\begin{align}
		\norm{Z_{\ovone} u}_{L^{2}}^{2}
		&= 4 \pi^{2} \int_{0}^{L} e^{6 a m (\cos \theta - \cos 3 \theta)}
			\frac{(\theta'(s))^{2}}{2 \kappa(\theta(s))} \cos^{2} \theta(s) \kappa(\theta(s)) \, d s \\
		&= 2 \pi^{2} \int_{0}^{2 \pi} \frac{1}{3 a \sqrt{10 - 6 \cos 2 \theta}}
			e^{6 a m (\cos \theta - \cos 3 \theta)} \cos^{2} \theta \, d \theta \\
		&\leq \frac{\pi^{2}}{3 a} \int_{0}^{2 \pi}
			e^{6 a m (\cos \theta - \cos 3 \theta)} \cos^{2} \theta \, d \theta.
	\end{align}
	Thus it is enough to prove
	\begin{equation}
	\label{eq:estimate-of-integral-for-epitrochoid}
		\int_{0}^{2 \pi} e^{6 a m (\cos \theta - \cos 3 \theta)} \cos^{2} \theta \, d \theta
		\leq \int_{0}^{2 \pi} e^{6 a m (\cos \theta - \cos 3 \theta)} \sin^{2} \theta \, d \theta.
	\end{equation}
	To this end,
	we consider
	\begin{equation}
		\int_{0}^{2 \pi} e^{6 a m (\cos \theta - \cos 3 \theta)} (\cos^{2} \theta - \sin^{2} \theta) \, d \theta
		= \int_{0}^{2 \pi} e^{6 a m (\cos \theta - \cos 3 \theta)} \cos 2 \theta \, d \theta
	\end{equation}
	and we to show that this integral is non-positive.
	Using the change of variables,
	we see that
	\begin{align}
		&\int_{0}^{2 \pi} e^{6 a m (\cos \theta - \cos 3 \theta)} \cos 2 \theta \, d \theta \\
		&= 2 \int_{0}^{\pi} e^{6 a m (\cos \theta - \cos 3 \theta)} \cos 2 \theta \, d \theta \\
		&= 4 \int_{0}^{\pi / 2} \cosh(6 a m (\cos \theta - \cos 3 \theta)) \cos 2 \theta \, d \theta \\
		&= 4 \int_{0}^{\pi / 4} \sbra*{\cosh(6 a m (\cos \theta - \cos 3 \theta))
			- \cosh(6 a m (\sin \theta + \sin 3 \theta))} \cos 2 \theta \, d \theta.
	\end{align}
	Here
	\begin{align}
		(\cos \theta - \cos 3 \theta)^{2} - (\sin \theta + \sin 3 \theta)^{2}
		&= \cos 2 \theta - 2 \cos 2 \theta + \cos 6 \theta \\
		&= 4 \cos 2 \theta (\cos^{2} 2 \theta - 1) \\
		&\leq 0
	\end{align}
	on $\clcl{0}{\pi / 4}$.
	We can deduce from this that
	\begin{equation}
		\cosh(6 a m (\cos \theta - \cos 3 \theta))
		\leq \cosh(6 a m (\sin \theta + \sin 3 \theta))
	\end{equation}
	on $\clcl{0}{\pi / 4}$.
	This and $\cos 2 \theta \geq 0$ on $\clcl{0}{\pi / 4}$ imply \cref{eq:estimate-of-integral-for-epitrochoid}.
	The latter statement follows from the former one and \cref{prop:criterion-of-embeddability}.
\end{proof}

\section{Concluding remarks}
\label{section:concluding-remarks}

The second author~\cite{Takeuchi2020-Paneitz}*{Theorem 1.1} proved that
the CR Paneitz operator on any \emph{embeddable} CR manifold is non-negative.
On the other hand,
we showed that the CR Paneitz operator on the three-dimensional torus has infinitely many negative eigenvalues
if $(M, T^{1, 0} M)$ is non-embeddable and $\kappa^{4} - \kappa \kappa'' + (\kappa')^{2} > 0$.
It is natural to ask what happens when the latter assumption is removed.

\begin{problem}
	Does the CR Paneitz operator on $(M, T^{1, 0} M, \theta)$ always have infinitely many negative eigenvalues
	if $(M, T^{1, 0} M)$ is non-embeddable?
\end{problem}

Case, Chanillo, and Yang~\cite{Case-Chanillo-Yang2015} introduced the supplementary space
\begin{equation}
	\calW
	\coloneqq \Ker P \cap \calP^{\perp}.
\end{equation}
\cite{Takeuchi2020-Paneitz}*{Theorem 1.1} implies that $\calW = 0$ for any embeddable CR manifold.
Moreover,
\cref{thm:kernel-of-CR-Paneitz} shows that $\calW = 0$
even if $\gamma$ is an iterate of a simple closed curve,
in which case $(M, T^{1, 0} M)$ is non-embeddable.
It is natural to investigate the case where $\gamma$ is an arbitrary smooth closed curve
or $(A, B) \neq (1, 1)$.
To the best of our knowledge,
no example with $\calW \neq 0$ is currently known.

\begin{problem}
	Is there an example of $(M, T^{1, 0} M)$ for which $\calW \neq 0$?
\end{problem}

\bibliography{my-reference,my-reference-preprint}

\end{document}